\documentclass[10pt,leqno]{article}

\usepackage[lite]{amsrefs}
\usepackage{filecontents}

\usepackage{latexsym,enumerate,pstricks}
  \usepackage[notref, notcite]{}
\usepackage{amssymb, amsfonts, eucal,amsmath,amsthm}

\newtheorem{theorem}{Theorem}[section]
\newtheorem{lemma}[theorem]{Lemma}
\newtheorem{statement}[theorem]{Statement}
\newtheorem{openproblem}[theorem]{Open Problem}
 \newtheorem{corollary}[theorem]{Corollary}

\newcommand{\thm}[1]{Theorem~\ref{#1}}
\newcommand{\lem}[1]{Lemma~\ref{#1}}
\newcommand{\cor}[1]{Corollary~\ref{#1}}

\newcommand{\sectio}[1]{Section~\ref{#1}}

 \DeclareMathOperator{\Lip}{Lip}
 \DeclareMathOperator{\dist}{dist}

\newcommand{\sect}[1]{\section{#1}  \setcounter{equation}{0}  }

\newcommand{\norm}[2]{\left\|#1\right\|_{#2}}
\newcommand{\ds}{\displaystyle}
\newcommand{\cl}{cl}

\newcommand{\Poly}{\Pi}
\newcommand{\Pn}{\Poly_n}

\newcommand{\E}{{\mathcal E}}

\newcommand{\A}{{\mathcal A}}
\newcommand{\B}{\mathcal B}
\newcommand{\NN}{{\mathcal N}}
\newcommand{\N}{\mathbb N}
\newcommand{\DD}{{\mathcal D}}
\newcommand{\andd}{\quad\mbox{\rm and}\quad}
\newcommand{\ttau}{\widetilde\tau}
\newcommand{\J}{{\mathcal{J}}}
\newcommand\w{{\omega}}
\newcommand{\C}{C}
\newcommand{\mon}{\Delta^{(1)}}

\def\be  {\begin{equation}}
\def\ee  {\end{equation}}

\newcommand{\ineq}[1]{{\rm(\ref{#1})}}

\newcommand{\ie}{{\em i.e., }}
\newcommand{\eg}{{\em e.g.}}
\newcommand{\st}{\;\; \big| \;\;}

\title{{\sc Interpolatory pointwise estimates for monotone polynomial approximation}}

\author{K. A.  Kopotun\thanks{Department of Mathematics, University of
Manitoba, Winnipeg, Manitoba, R3T 2N2, Canada ({\tt kirill.kopotun@umanitoba.ca}).  Supported by NSERC of Canada Discovery Grant RGPIN 04215-15.} \and
D. Leviatan\thanks{Raymond and Beverly Sackler School of Mathematical
Sciences, Tel Aviv University, Tel Aviv 69978, Israel ({\tt leviatan@post.tau.ac.il}).}
\and
I. A. Shevchuk\thanks
{Faculty of Mechanics and Mathematics, Taras
Shevchenko National University of Kyiv, 01601 Kyiv, Ukraine ({\tt shevchuk@univ.kiev.ua}).}
}

\begin{document}

 \maketitle

\begin{abstract}
Given a nondecreasing function $f$ on $[-1,1]$, we investigate how well it can be approximated by nondecreasing algebraic polynomials that interpolate it at $\pm 1$. We establish pointwise estimates of the approximation error  by such polynomials that yield interpolation at the endpoints (\ie the estimates become zero at $\pm 1$).
We call such estimates ``interpolatory estimates''.

In 1985, DeVore and Yu were the first to obtain this kind of results for monotone polynomial approximation.
 Their estimates involved the second modulus of smoothness $\omega_2(f,\cdot)$ of $f$ evaluated at $\sqrt{1-x^2}/n$ and were valid for all $n\ge1$.
 The current paper is devoted to proving that if $f\in C^r[-1,1]$, $r\ge1$, then the interpolatory estimates are valid for the second modulus of smoothness of $f^{(r)}$, however, only for $n\ge\NN$ with $\NN= \NN(f,r)$, since it is known that such estimates are in general invalid with $\NN$ independent of $f$.

Given a number $\alpha>0$, we write $\alpha=r+\beta$ where $r$ is a nonnegative integer and $0<\beta\le1$, and denote by $\Lip^*\alpha$ the class of all functions $f$ on $[-1,1]$ such that $\w_2(f^{(r)}, t) = O(t^\beta)$. Then, one  important corollary of the main theorem in this paper is the following result that  has been an open problem for $\alpha\geq 2$ since 1985:
\begin{quote}
If $\alpha>0$, then a function $f$ is nondecreasing and in $\Lip^*\alpha$, if and only if, there exists a constant $C$ such that,  for all sufficiently large $n$, there are nondecreasing polynomials $P_n$, of degree $n$, such that
\[
|f(x)-P_n(x)| \leq C \left( \frac{\sqrt{1-x^2}}{n}\right)^\alpha, \quad x\in [-1,1].
\]
\end{quote}

\end{abstract}

%

\sect{Introduction and main results}

Given a nondecreasing function $f$ on $[-1,1]$ and a set $\Xi := \{\xi_i\}_{i=1}^m \subset [-1,1]$ ($\xi_i\neq \xi_j$ if $i\neq j$), is there a nondecreasing algebraic polynomial that not only approximates $f$  well but also interpolates $f$ at the points in $\Xi$?
For a general set $\Xi$, the answer is clearly ``no''. If $m\geq 3$, then the nondecreasing interpolating polynomial may not exist at all (consider $f$ which is constant on $[\xi_1,\xi_2]$ and such that $f(\xi_3)>f(\xi_2)$).

If $m=1$, then the case for interpolation at either $-1$ or $1$ (but not both) was considered in \cite{GLSW}, and we leave the discussion of the case when $-1 < \xi_1 < 1$ for another time.

Finally, if $m=2$, then the nondecreasing polynomial interpolating $f$ at $\xi_1$ and $\xi_2$  exists, but it does not   approximate $f$ well at all if $[\xi_1,\xi_2] \neq [-1,1]$ (again, consider $f$ which is constant on $[\xi_1,\xi_2]$ and is strictly increasing outside this interval). Hence, for $m=2$, the only non-trivial case that remains is when the nondecreasing polynomial interpolates $f$ at the endpoints of $[-1,1]$.
 We call the pointwise estimates of the degree of approximation of $f$ by such polynomials that yield interpolation at the endpoints (\ie the estimates become zero at $\pm 1$) ``interpolatory estimates in monotone polynomial approximation''.

We also note that the situation with strictly increasing functions is rather different (see \eg \cites{iliev, pr} and the references therein), since for any strictly increasing function $f$ and any collection of points $\Xi$, there exists a strictly increasing polynomial of a sufficiently large degree that interpolates $f$ at all points in $\Xi$. How well this polynomial approximates $f$ is an interesting problem but we do not consider it in this manuscript.

More discussions of various related results  on monotone approximation can be found in our survey paper \cite{KLPS}.

For $r\in\N$, let $C^r[a,b]$, $-1\le a<b\le1$, denote the space of $r$ times continuously differentiable functions on $[a,b]$, and let $C^0[a,b]=C[a,b]$ denote the space of continuous functions on $[a,b]$, equipped with the uniform norm $\|\cdot\|_{[a,b]}$.

For $f\in C[a,b]$ and any $k\in\N$, set
\[
\Delta^k_u(f,x;[a,b]):=
\begin{cases} \ds \sum_{i=0}^k(-1)^i  \binom ki    f(x+(k/2-i)u),&\quad x\pm (k/2)u\in[a,b] , \\
0,&\quad{\rm otherwise},
\end{cases}
\]
and denote by
$$
\omega_k(f,t;[a,b]):=\sup_{0<u\le t}\|\Delta^k_u(f,\cdot;[a,b])\|_{[a,b]}
$$
its $k$th modulus of smoothness.
When dealing with $[a,b]=[-1,1]$, we suppress referring to the interval, that is, we denote  $\|\cdot\|:=\|\cdot\|_{[-1,1]}$ and $\omega_k(f,t):=\omega_k(f,t;[-1,1])$.

Finally, let
\be\label{varphi}
\varphi(x):=\sqrt{1-x^2}\quad\text{and}\quad\rho_n(x):=\frac{\varphi(x)}n+\frac1{n^2},
\ee
and denote by $\mon$ the class of all nondecreasing functions on $[-1,1]$, and by $\Pn$   the space of algebraic polynomials of degree $\le n$.

In 1985, DeVore and Yu \cite{DY}*{Theorem 1}  proved that, for  $f\in\C[-1,1]\cap\mon$ and any $n\in\N$, there exists a   polynomial $P_n\in\Pn\cap\mon$ such that
\be \label{dyineq}
|f(x)-P_n(x)|\le c\omega_2\left(f,\frac{\varphi(x)}n\right),\quad x\in[-1,1],
\ee
where $c$ is an absolute constant.

In 1998, it was proved in \cite{LS98}*{Theorem 4}   that there exists  $f\in \C[-1,1]\cap \mon$ such that
\be \label{lsineq}
\limsup_{n\to\infty} \inf_{P_n\in\Pn\cap\mon} \max_{x\in [-1,1]} \frac{|f(x)-P_n(x)|}{\w_3(f, \rho_n(x))} = \infty ,
\ee
which implies that $\w_2$ in  \ineq{dyineq}   cannot be replaced by $\w_3$ even if the constant $c$ and how large $n$ is are allowed to depend on the function $f$.

If the function $f$ is smoother, then the following is valid (see \cite{S}):
\begin{quote}
For any $k,r\in\N$ and $f\in\C^r[-1,1]\cap\mon$,   there exists a sequence  of polynomials $P_n\in\Pn\cap \mon$ such that, for every $n\geq k+r-1$ and each $x\in[-1,1]$, we have
\[
|f(x)-P_n(x)|\le c(k,r)\rho_n^r(x)\omega_k(f^{(r)},\rho_n(x)) .
\]
\end{quote}

A natural question now is whether  \ineq{dyineq} may be strengthened for functions having higher smoothness.
More precisely, the following problem  needs to be resolved: find all values of $k\in\N$ and $r\in\N_0$ such that the following statement is true, and investigate whether or not the number $\NN$ in this statement has to depend on $f$.

 \begin{statement} \label{stat}
For every $f\in\C^r[-1,1]\cap\mon$, $r\ge1$, there exist a number $\NN\in\N$ and a sequence $\{P_n\}_{n=\NN}^\infty$ of polynomials $P_n\in\Pn\cap \mon$ such that, for every $n\geq \NN$ and each $x\in[-1,1]$, we have
 \be \label{problemineq}
|f(x)-P_n(x)|\le c(k,r)\left(\frac{\varphi(x)}n\right)^r\omega_k\left(f^{(r)},\frac{\varphi(x)}n\right) .
 \ee
 \end{statement}
In view of \ineq{dyineq} and \ineq{lsineq}, Statement~\ref{stat} is true if $k+r \leq 2$ (with $\NN=1$) and is not true for $r=0$ and $k\geq 3$.

Using  the same method as was used to prove  \cite{GLSW}*{Theorem 4}  one can show that, for any $r\in\N$ and each $n\in\N$, there is a function $f\in\C^r[-1,1]\cap\mon$, such that for every polynomial $P_n\in\Pn\cap\mon$ and any positive on $(-1,1)$ function $\psi$ such that $\lim_{x\to \pm 1} \psi(x)=0$, either
\be \label{glswineq}
\limsup_{x\to -1} \frac{|f(x)-P_n(x)|}{\varphi^2(x) \psi(x)} = \infty \quad \mbox{\rm or}\quad
\limsup_{x\to 1} \frac{|f(x)-P_n(x)|}{\varphi^2(x)\psi(x)} = \infty .
\ee
In particular, this implies that Statement~\ref{stat} is not valid with $\NN$ independent of $f$ if $k+r\geq 3$. However, in this paper, we show that this statement is valid for $k=2$ and any $r\in\N$ provided that $\NN$ depends on $f$. Namely, the following theorem is the main result in this manuscript.

\begin{theorem}\label{thm1} Given $r\in\N$, there is a constant $c=c(r)$ with the property that if $f\in C^r[-1,1]\cap\mon$, then there exists a number $\NN=\NN(f,r)$, depending on $f$ and $r$, such that for every $n\ge \NN$, there is  $P_n\in \Pn \cap \mon$  satisfying
\be\label{interpol}
|f(x)-P_n(x)|\le c(r)\left(\frac{\varphi(x)}n\right)^r\omega_2\left(f^{(r)},\frac{\varphi(x)}n\right),\quad x\in[-1,1] .
\ee
Moreover, for $x\in \left[-1, -1+ n^{-2}\right] \cup \left[1-n^{-2}, 1\right]$ the following stronger estimate is valid:
\be\label{interpol1}
|f(x)-P_n(x)|\le c(r)\varphi^{2r}(x)\omega_2\left(f^{(r)},\frac{\varphi(x)}n\right).
\ee
\end{theorem}

Given a number $\alpha>0$, we write $\alpha=r+\beta$ where $r$ is a nonnegative integer and $0<\beta\le1$. Denote by $\Lip^*\alpha$ the class of all functions $f$ on $[-1,1]$ such that $\w_2(f^{(r)}, t) = O(t^\beta)$.

An immediate corollary of \thm{thm1} and the classical (Dzyadyk) converse theorems for approximation by algebraic polynomials is the following result on characterization of $\Lip^*\alpha$.

\begin{corollary} \label{maincor}
If $\alpha>0$, then a function $f$ is nondecreasing and in $\Lip^*\alpha$, if and only if, there exists a constant $C$ such that, for sufficiently large $n$, there are nondecreasing polynomials $P_n$ of degree $n$ such that
\[
|f(x)-P_n(x)| \leq C \left( \frac{\sqrt{1-x^2}}{n}\right)^\alpha, \quad x\in [-1,1].
\]
\end{corollary}
Note that, for $0<\alpha<2$, \cor{maincor} follows from \ineq{dyineq} (and was stated   in \cite{DY}).

In order to state another  corollary of \thm{thm1} we recall that $W^r$ denotes the space of $(r-1)$ times continuously differentiable functions on $[-1,1]$ such that $f^{(r-1)}$ is absolutely continuous in $(-1,1)$ and
$\norm{f^{(r)}}{\infty} < \infty$,
where $\norm{\cdot}{\infty}$ denotes the essential supremum  on $[-1,1]$.

\begin{corollary} \label{secondcor}
For any   $f\in W^r \cap\mon$, $r\in\N$, there exists a number
$\NN=\NN(f,r)$,   such that for every $n\ge \NN$,
\[
\inf_{P_n\in \Pn \cap \mon} \norm{ \frac{f-P_n}{\varphi^r} }{\infty} \leq \frac{c(r)}{n^r} \norm{f^{(r)}}{\infty}  .
\]
\end{corollary}
Note that, for $r\le 2$, \cor{secondcor}   follows from \ineq{dyineq} with $\NN=1$.

The paper is organized as follows. In Section~\ref{section2}, we introduce various notations that are used throughout the paper. Several inequalities for the Chebyshev partition are discussed in Section~\ref{aux}, and Section~\ref{auxresults} is devoted to a discussion  of polynomial approximation of indicator functions. In Section~\ref{section5}, we prove several auxiliary results on various properties of piecewise polynomials. We need those since our proof of \thm{thm1} will be based on approximating $f$ by certain monotone piecewise polynomial functions, and then approximating these functions by monotone polynomials. In Section~\ref{section6}, we discuss approximation of monotone piecewise polynomials with ``small'' first derivatives by monotone polynomials. Section~\ref{section7} is devoted to constructing a certain partition of unity. Simultaneous polynomial approximation of piecewise polynomials and their derivatives is discussed in Section~\ref{section8}  and, in Section~\ref{yetanother}, we construct one particular polynomial with controlled first derivative. Finally, in Section~\ref{sec5}, we use all these auxiliary results to prove a lemma on monotone polynomial approximation of piecewise polynomials that is then used in Section~\ref{sec555} to prove \thm{thm1}.


We conclude this section by stating the following open problem.

\begin{openproblem} Find all pairs $(r,k)$ with $r\in\N$ and $k\geq 3$ for which  Statement~$\ref{stat}$
 is valid (with $\NN$ dependent on $f$).
\end{openproblem}


\sect{Notations} \label{section2}

Recall that the Chebyshev partition of $[-1,1]$ is the ordered set
$X_n:=(x_{j})_{j=0}^n$, where
\[
x_j := x_{j,n} := \cos(j\pi/n) , \quad 0\leq j \leq n.
\]
We refer to $x_j$'s as ``Chebyshev knots'' and note that $x_j$'s are the extremum points of the Chebyshev polynomial of the first kind of degree $n$.
It is also convenient to denote $x_j:= x_{j,n}:=1$ for $j<0$ and $x_j:=x_{j,n}:=-1$  for $j>n$.
Also, let
$I_j := [x_{j},x_{j-1}]$,
$h_j :=   |I_{j}| := x_{j-1}-x_j$, and
\[
\chi_j(x) := \chi_{[x_j, 1]}(x) =
\begin{cases}
1, & \mbox{\rm if } x_j \leq x \leq 1, \\
0, & \mbox{\rm otherwise}.
\end{cases}
\]
Denote by $\Sigma_{k}:=\Sigma_{k,n}$ the set of all right continuous piecewise polynomials of degree $\leq k-1$ with knots at $x_j$, $1\leq j \leq n-1$.
That is,
\[
S\in \Sigma_{k} \quad \mbox{\rm if and only if} \quad    S|_{[x_j, x_{j-1})} \in\Poly_{k-1} ,  \;  2\le j\le n, \andd S|_{[x_1, 1]} \in\Poly_{k-1} .
\]
Throughout this paper, for $S\in\Sigma_k$,  we denote the polynomial piece of $S$ inside the interval $I_j$ by $p_j$, \ie
\[
p_j :=  p_j(S) :=  S|_{[x_j, x_{j-1})} , \quad 2\le j\le n , \andd p_1 :=  p_1(S) :=  S|_{[x_1, 1]} .
\]
For $k\in\N$, let $\Phi^k$ be the class of all ``$k$-majorants'', \ie continuous nondecreasing functions $\psi$ on $[0,\infty)$ such that $\psi(0)=0$ and $t^{-k} \psi(t)$ is nonincreasing on $[0,\infty)$.
In other words,
\[
\Phi^k = \left\{ \psi\in C[0,\infty)  \;  \big| \;  \psi\uparrow, \;  \psi(0)=0, \; \text{and}\;   t_2^{-k} \psi(t_2) \leq t_1^{-k} \psi (t_1)
   \; \mbox{\rm for $0<t_1\leq t_2$} \right\} .
\]
Note that, given $f\in\C^r[-1,1]$,  while the  function $\phi(t) := t^r \w_k(f^{(r)}, t)$ does not have to be in $\Phi^{k+r}$, it is equivalent to a function from $\Phi^{k+r}$. Namely,
$\phi(t) \leq \phi^*(t) \leq 2^{k} \phi(t)$,
where $\phi^*(t) := \sup_{u>t} t^{k+r} u^{-k-r} \phi(u) \in \Phi^{k+r}$ (see, \eg, \cite{DS}*{p. 202}).

For $1\le i,j\le n$, let
\[
I_{i,j}:=\bigcup_{k=\min\{i,j\}}^{\max\{i,j\}}I_k = \left[ x_{\max\{i,j\}}, x_{\min\{i,j\}-1} \right]
\]
and
\[
h_{i,j}:=|I_{i,j}|=\sum_{k=\min\{i,j\}}^{\max\{i,j\}}h_k = x_{\min\{i,j\}-1} - x_{\max\{i,j\}} .
\]
In other words, $I_{i,j}$ is the smallest interval that contains both $I_i$ and $I_j$, and $h_{i,j}$ is its length.

For $\phi\in\Phi^k$, which is not identically zero (otherwise everything is either trivial or of no value), and $S\in\Sigma_k$, denote
\be\label{bij}
b_{i,j}(S,\phi):=
\frac{\|p_i-p_j\|_{I_i}}{\phi(h_j)}\left(\frac{h_j}{h_{i,j}}\right)^k,\quad1\le i,j\le n .
\ee
(Note that $b_{i,j}(S,\phi)=a_{i,j}(S)/{\phi(h_j)}$ with $a_{i,j}$  defined in  \cite{LS2002}*{(6.1)}.)

Also, for   $S\in\Sigma_{k}$ and an interval $A\subseteq [-1,1]$ containing at least one interval $I_\nu$, denote
\[
b_k(S,\phi, A)  :=\max_{1\leq i,j\leq n} \left\{ b_{i,j}(S,\phi) \st I_i\subset  A\andd I_j\subset  A\right\} ,
\]
and
\[
b_k(S,\phi) :=  b_{k}(S,\phi, [-1,1]) =  \max_{1\le i,j\le n}b_{i,j}(S,\phi).
\]

Throughout this paper, we reserve the notation ``$c$'' for positive constants that are either absolute or may only depend on the parameter $k$ (and eventually will depend on $r$). We use the notation ``$C$'' and ``$C_i$'' (the latter only in \sectio{sec5}) for all other positive constants and indicate in each section the parameters that they may depend on.
All constants $c$ and $C$ may be different on different occurrences (even if they appear in the same line), but the indexed constants $C_i$ are fixed throughout   \sectio{sec5}.


\sect{Inequalities for the Chebyshev partition}\label{aux}

In this section,  we collect all the facts and inequalities for the Chebyshev partition that we need throughout this paper.

It is rather well known (see, \eg, \cite{DS}*{pp. 382-383, 408}) and not too difficult to verify that
\begin{align} \label{rho}
\frac{\varphi(x)}n<\rho_n(x)<h_{j}&<5\rho_n(x),\quad x\in I_{j},\quad1\le j\le n,\\ \nonumber
h_{j\pm 1}&<3h_{j},\quad1\le j\le n,
\end{align}
and
\begin{align}\label{rho1}
\rho^2_n(y)&<4\rho_n(x)(|x-y|+\rho_n(x)) \andd \\
(|x-y|+\rho_n(x))/2&<(|x-y|+\rho_n(y))<2(|x-y|+\rho_n(x)) ,\quad x,y\in[-1,1]\nonumber.
\end{align}
(We remark that the inequalities on the second line in \ineq{rho1} immediately follow from the estimate on the first line.)

Also, we observe that
\be \label{newauxest}
 \rho_n(x)  \le  |x-x_j|,\quad \text{for any }\; 0\leq j\leq n \andd  x\notin(x_{j+1}, x_{j-1}) .
\ee
Indeed, \ineq{newauxest} holds for $x=x_{j \pm 1}$ by \ineq{rho}, and for all other $x\notin(x_{j+1}, x_{j-1})$, it follows from the inequalities
$x+\rho_n(x) \leq x_{j+1}+\rho_n(x_{j+1})$ if $x<x_{j+1}$, and
$x-\rho_n(x) \geq x_{j-1}-\rho_n(x_{j-1})$ if $x> x_{j-1}$, that can be verified directly or using the fact that   $x +\rho_n(x)$ increases on $\left[-1, n/\sqrt{n^2+1}\right] \supset [-1, x_1]$ and $x -\rho_n(x)$ increases on $\left[-n/\sqrt{n^2+1}, 1\right] \supset [x_{n-1}, 1]$.

Now, denote
\[
 \psi_j :=\psi_j(x) :=  \frac{|I_j|}{|x-x_j|+|I_j|}
\andd
\delta_n(x):= \min\{ 1, n\varphi(x)\} ,
\quad x\in [-1,1],
\]
and note that
\[
\delta_n(x)=1 \quad \text{if}\quad x\in[x_{n-1},x_1]
\]
and
\[
\delta_n(x)\le n\varphi(x)< \pi\delta_n(x)  \quad \text{if}\quad x\in[-1,x_{n-1}]\cup[x_1,1].
\]

It follows from \ineq{rho} and \ineq{rho1} that
\be \label{anotherauxest}
\rho_n^2(x) < 4 \rho_n(x_j) \left(|x-x_j|+\rho_n(x_j) \right) < 8 h_j \left(|x-x_j|+\rho_n(x) \right) ,
\ee
and thus
\be \label{auxsum}
\left( \frac{\rho_n(x)}{\rho_n(x) + |x-x_j|} \right)^2 <   \frac{8h_j}{\rho_n(x) + |x-x_j|} < c \psi_j(x) .
\ee
Similarly, \ineq{rho} and \ineq{rho1} imply (see, \eg, \cite{K-sim}*{(26)}) that
\be \label{hjrho}
c\psi_j^{2}(x)\rho_n(x) \leq \rho_n(x_j)\leq c\psi_j^{-1}(x)\rho_n(x), \quad 1\leq j \leq n \andd x\in [-1,1],
\ee
where $c$ are some absolute constants.

It is not difficult to see that, for all $1\leq j\leq n$ and $x\in [-1,1]$,
\be\label{distance1}
\rho_n(x) + \dist(x, I_j) \leq  \rho_n(x) + |x-x_j|  \leq 16 \left( \rho_n(x) + \dist(x, I_j) \right) .
\ee
Indeed, the first inequality in \ineq{distance1} is obvious, and the second   follows from
\[
 |x-x_j| \leq 4\dist(x, I_j) + 15\rho_n(x) ,
\]
which is verified using \ineq{rho} and separately considering the cases $x\in I_{j-1}\cup I_j \cup I_{j+1}$ and $x\not\in I_{j-1}\cup I_j \cup I_{j+1}$ (in the latter case,
 there is at least one interval $I_i$, $i\neq j$, between $x$ and $I_j$,  so that $|x-x_j| \leq h_j +  \dist(x, I_j)   \leq 4 \dist(x, I_j)$).

Also, it is straightforward to check   that
\be \label{sumpsi}
\sum_{j=1}^{n}\psi_j^{2}(x) \leq c , \quad x\in [-1,1] ,
\ee
and so, by virtue of \ineq{distance1} and \ineq{auxsum},
\be \label{sumest}
\sum_{j=1}^{n} \left( \frac{\rho_n(x)}{\rho_n(x) + \dist(x, I_j)} \right)^4 \leq c .
\ee

In order to quote several results from \cite{K-sim} in the form used in this paper we need the following observation.
First,
it is known (see, \eg, \cite{K-sim}*{Proposition 4}) that
\[
\frac{1-x^2}{ (1+x_{j-1})(1-x_j)} \leq  2\psi_j^{-2}(x),   \quad 1\leq j \leq n \andd -1\leq x\leq 1 .
\]
Now, since
\[
\min_{1\leq j \leq n} \left\{ (1+x_{j-1})(1-x_j) \right\}
\geq 1-x_1 \geq 2/n^2 ,
\]
we have
\[
\frac{1-x^2}{ (1+x_{j-1})(1-x_j)} \leq \frac{n^2 \varphi^2(x)}{2}  ,
\]
and hence, for all $1\leq j \leq n$ and $x\in [-1,1]$,
\be \label{delta}
\frac{1-x^2}{(1+x_{j-1})(1-x_j)} \leq 2 \min\{1, n^2 \varphi^2(x) \} \psi_j^{-2}(x) = 2 \delta_n^2(x) \psi_j^{-2}(x) .
\ee
Conversely, by \ineq{hjrho}
\be \label{delta1}
\frac{1-x^2}{(1+x_{j-1})(1-x_j)}\geq  \frac{c\varphi^2(x)}{ n^2\rho_n^2(x_j)}\geq c\psi_j^{2}(x) \frac{n^2\varphi^2(x)}{(n\varphi(x)+1)^2}
\geq c\psi_j^{2}(x)\delta_n^2(x),
\ee
where the first inequality is valid since
\[
(1+x_{j-1})(1-x_j) = 1-x_j^2 + h_j(1-x_j) \leq n^2\rho_n^2(x_j) + \rho_n (x_j) \le 2 n^2\rho_n^2(x_j).
\]

\sect{Auxiliary results on polynomial approximation of indicator functions} \label{auxresults}
All constants $C$ in this section depend on $\alpha$ and $\beta$.
\begin{lemma}\label{lem5}
Given $\alpha,\beta \geq 1$, there exist polynomials $\tau_j$, $1\leq j\leq n-1$,
of degree $\le Cn$  satisfying, for all $x\in[-1,1]$,
\be \label{taudoubleprime}
\tau_j'(x)\ge C|I_j|^{-1}\delta_n^{8\alpha}(x)\psi_j^{30(\alpha+\beta)}(x),
\ee
\be \label{derivatives}
\left|\tau_j^{(q)}(x)\right|\leq C|I_j|^{-q}\delta_n^\alpha(x)\psi_j^\beta(x),\quad 1\leq q\leq \alpha ,
\ee
and
\be\label{tauj}
|\chi_j(x)-\tau_j(x)|\le
C\delta_n^\alpha(x) \psi_j^\beta(x).
\ee
\end{lemma}

\begin{proof}
First, estimates \ineq{derivatives} and \ineq{tauj} immediately follow from \cite{K-sim}*{Lemma 6} taking into account \ineq{delta} and setting
$\mu := \lceil 10\alpha+10\beta \rceil$ and $\xi := \lceil 3\alpha \rceil$ in that lemma.
Estimate \ineq{taudoubleprime} was not proved in \cite{K-sim}, and so, even though its proof is very similar to that of \ineq{derivatives} and \ineq{tauj}, we adduce it here for the sake of completeness.

Recall the definition of polynomials $\tau_j$:
\be \label{deftau}
\tau_j(x) = d_j^{-1} \int_{-1}^x (1-y^2)^\xi t_j^\mu(y)\, dy ,
\ee
where
\be \label{deftj}
t_j(x) := \left( \frac{\cos 2n\arccos x}{x-x_j^0} \right)^2 + \left( \frac{\sin 2n\arccos x}{x-\bar x_j} \right)^2 ,
\ee
$\bar x_j := \cos((j-1/2)\pi/n)$ for $1\leq j \leq n$,
$x_j^0 := \cos((j-1/4)\pi/n)$ for $1\leq j <n/2$,
$x_j^0 := \cos((j-3/4)\pi/n)$ for $n/2\leq j \leq n$, and the normalizing
constants $d_j$ are chosen so that $\tau_j(1)=1$.

It is known (see, \eg, \cite{K-sim}*{(22), Proposition 5}) and is not difficult to prove that
\be \label{behtj}
 t_j(x) \sim (|x-x_j|+h_j)^{-2} , \quad x\in [-1,1] \andd 1\leq j \leq n,
\ee
and
\[
d_j  \sim (1+x_{j-1})^{\xi} (1-x_{j})^{\xi} h_j^{-2\mu+1}, \quad \text{if }\; \mu\geq \xi+1.
\]
Here and later, by $X\sim Y$ we mean that there exists a positive constant $c$ (independent of the important parameters) such that $c^{-1} X \le Y \le c X$.

 Hence, using \ineq{delta1}, we have
\begin{align*}
\tau_j'(x) &=  d_j^{-1} (1-x^2)^\xi t_j^\mu(x)\\
& \geq   C \frac{h_j^{2\mu-1}}{(1+x_{j-1})^\xi(1-x_{j})^\xi} (1-x^2)^\xi (|x-x_j|+h_j)^{-2\mu} \\
& \geq
C   h_j^{-1} \delta_n^{2\xi}(x) \psi_j^{2\mu+2\xi}(x) \\
&\geq
C   h_j^{-1} \delta_n^{8\alpha}(x) \psi_j^{30(\alpha+\beta)}(x) .
\end{align*}

\end{proof}

\begin{lemma}\label{lemnew5}
Given $\alpha,\beta >0$, there exist polynomials $\ttau_j$, $1\leq j\leq n-1$,
of degree $\le Cn$  satisfying
\be \label{decreasing}
\ttau_j'(x) \leq 0, \quad \mbox{\rm for }\; x\in [-1, x_j] \cup [x_{j-1}, 1] ,
\ee
and, for all $x\in[-1,1]$,
\be \label{newderivatives}
\left|\ttau'_j(x) \right| \leq C|I_j|^{-1} \delta_n^\alpha(x) \psi_j^\beta(x)
\ee
and
\be\label{newtauj}
|\chi_j(x)-\ttau_j(x)|\le
C\delta_n^\alpha(x) \psi_j^\beta(x).
\ee
\end{lemma}

\begin{proof}
We let
\[
\ttau_j(x) := \widetilde d_j^{-1} \int_{-1}^x (y-x_j)(x_{j-1}-y)(1-y^2)^\xi t_j^\mu(y)\, dy
\]
with $t_j$ defined in \ineq{deftj} and $\widetilde d_j$ is so chosen that $\ttau_j(1)=1$, and where $\xi$ and $\mu$ are sufficiently large and will be prescribed later.
Clearly, \ineq{decreasing} is satisfied.

It is possible to show (see, \eg, \cite{K-coconvex}*{Proposition 4} with $m=k=\xi+1$, $a_1=\dots=a_{m-1}=-1$, $b_1=\dots=b_{k-1}=1$, $a_m=x_j$, $b_k=x_{j-1}$) that
\[
\widetilde d_j \sim (1+x_{j-1})^\xi (1-x_j)^\xi h_j^{-2\mu+3} , \quad \mbox{\rm if }\; \mu \geq 10\xi+15 .
\]
Hence, using \ineq{behtj} we have
\begin{eqnarray*}
\left|\ttau'_j(x) \right| &=& \widetilde d_j^{-1} (1-x^2)^\xi |x-x_j| |x_{j-1}-x| |t_j^\mu(x)| \\
& \leq &
C \left( \frac{1-x^2}{ (1+x_{j-1})(1-x_j)} \right)^\xi h_j^{-1} \psi_j^{2\mu-2}(x) .
\end{eqnarray*}
We  note (cf. \cite{K-sim}*{(25)}) that, for all $x\in [-1,1]$,
\[
\frac{1+x }{1+x_{j-1}} \leq c \psi_j^{-1}(x) \andd \frac{1-x}{ 1-x_{j}} \leq c \psi_j^{-1}(x) .
\]
Now, if $x < x_j$, then
\begin{eqnarray*}
\lefteqn{ |\chi_j(x)-\ttau_j(x)|   =   |\ttau_j(x)| =  \left| \int_{-1}^x  \ttau_j'(y)  dy \right| } \\
& \leq &
C  h_j^{-1} \int_{-1}^x \left( \frac{1+y}{  1+x_{j-1} } \right)^\xi \left( \frac{h_j }{ |y-x_j|+h_j} \right)^{2\mu-\xi-2} dy \\
& \leq &
C \left( \frac{1+x  }{  1+x_{j-1} } \right)^\xi h_j^{2\mu-\xi-3} \int_{-\infty}^x (x_j-y + h_j)^{-2\mu+\xi+2} dy \\
& \leq &
C \left( \frac{1-x^2}{ (1+x_{j-1})(1-x_j)} \right)^\xi \psi_j^{2\mu-\xi-3} .
\end{eqnarray*}
Similarly, for $x\geq x_j$, we write
\begin{eqnarray*}
\lefteqn{ |\chi_j(x)-\ttau_j(x)|   =   |1-\ttau_j(x)| =  \left| \int_{x}^1 \ttau_j'(y)  dy \right| } \\
& \leq &
C  h_j^{-1} \int_{x}^1 \left( \frac{1-y }{  1-x_{j} } \right)^\xi \left( \frac{h_j }{ |y-x_j|+h_j} \right)^{2\mu-\xi-2} dy \\
& \leq &
C \left( \frac{1-x }{  1-x_{j} } \right)^\xi  h_j^{2\mu-\xi-3} \int_{x}^\infty (y-x_j + h_j)^{-2\mu+\xi+2} dy \\
& \leq &
C \left( \frac{1-x^2 }{ (1+x_{j-1})(1-x_j)} \right)^\xi \psi_j^{2\mu-\xi-3} .
\end{eqnarray*}
 Finally, using \ineq{delta}, we conclude that
 \[
 \left|\ttau'_j(x) \right| \leq C \delta_n^{2\xi} h_j^{-1} \psi_j^{2\mu-2\xi-2}(x)
 \]
and
\[
|\chi_j(x)-\ttau_j(x)| \leq C \delta_n^{2\xi} h_j^{-1} \psi_j^{2\mu-3\xi-3}(x) ,
\]
and it is enough to set $\xi := \lceil \alpha/2 \rceil $ and $\mu := \lceil \beta + 5\alpha \rceil + 25$ in order to complete the proof.
\end{proof}

\sect{Auxiliary results on properties of piecewise polynomials} \label{section5}

All constants $c$ in this section depend only on $k$.

 The following lemma is valid (compare with \cite{DLS}*{Lemma 1.4}).

\begin{lemma}\label{b_k<c}
Let $k\in\N$, $\phi\in\Phi^k$, $f\in C[-1,1]$ and $S\in\Sigma_{k,n}$. If
\[
\omega_k(f,t)\le\phi(t)
\]
and
\be\label{piecewise}
|f(x)-S(x)|\le \phi(\rho_n(x)),\quad x\in[-1,1],
\ee
then
\[
b_k(S, \phi)\le c.
\]
\end{lemma}

\begin{proof}
Recall that $\phi$ is not identically zero, so that $\phi(x)>0$, $x>0$. For $1\le i,j\le n$, we have
\[
b_{i,j}(S,\phi)\le\frac{\|p_i-f\|_{I_i}}{\phi(h_j)}\left(\frac{h_j}{h_{i,j}}\right)^k+
\frac{\|f-p_j\|_{I_i}}{\phi(h_j)}\left(\frac{h_j}{h_{i,j}}\right)^k=:\sigma_1+\sigma_2.
\]
Now,  we note that, for any $1\leq \nu\leq n$, inequalities \ineq{piecewise} and \ineq{rho} imply
\[
\norm{p_\nu-f}{I_\nu} \le   \norm{ \phi(\rho_n)}{I_\nu} \le \phi(h_\nu) .
\]
Hence, $\sigma_1\le1$, where we used the fact that if $h_i\le h_j$, then $\phi(h_i)\le\phi(h_j)$, and if $h_i>h_j$, then
$\phi(h_i)/\phi(h_j) \leq  h^k_i/h^k_j$.

In order to estimate $\sigma_2$, we first recall the following estimate (see \cite{DS}*{(6.17), p. 235}). For any $g\in\C[-1,1]$, $k\in\N$, $a\in [-1,1]$ and $h>0$ such that $a+(k-1)h \in [-1,1]$,
\[
|g(x)| \leq c \left(1+\frac{|x-a|}{h} \right)^k \left( \w_k(g, h) + \norm{g}{[a, a+(k-1)h]} \right) , \quad x\in [-1,1] .
\]
Setting $g := f-p_j$, $a:=x_j$ and $h:=h_j/\max\{1,k-1\}$, and observing that $\w_k(g, h) = \w_k(f-p_j, h) = \w_k(f, h) \leq \phi(h)$, we get
\[
|f(x)-p_j(x)| \leq c \left(1+\frac{|x-x_j|}{h_j} \right)^k \left( \phi(h_j) + \norm{f-p_j}{I_j} \right) , \quad x\in [-1,1] ,
\]
and so
\[
\norm{f-p_j}{I_i} \leq c \left(\frac{h_{i,j}}{h_j}\right)^k \phi(h_j).
\]
Hence, $\sigma_2\le c$, and the proof is complete.
\end{proof}

The next lemma, although claims a different inequality than \cite{DLS}*{Lemma 2.1}, is proved along the same lines. We bring its proof for the sake of completeness.

\begin{lemma}\label{important}
Let $k\in\N$, $\phi\in\Phi^k$ and $S\in\Sigma_{k,n}\cap\C[-1,1]$. Then
\be \label{inlem5}
b_k(S,\phi)\le\,c \norm{ \frac{\rho_n S'}{\phi(\rho_n)} }{\infty}.
\ee
\end{lemma}

\begin{proof}
We note that in the case $k=1$, the statement of the lemma is trivial since  $\Sigma_{1,n}\cap\C[-1,1]=\Poly_0$, and so both sides of \ineq{inlem5} are identically zero.
Hence, we assume that $k\geq 2$, and we may also assume that
\be\label{bound}
\norm{\frac{\rho_n S'}{\phi(\rho_n )}}{\infty}=1.
\ee
Since
\[
p_j(x)=S (-1)+\int_{-1}^{x_j}S'(u)du+\int_{x_j}^xp'_j(u)du,\quad 1\leq j \leq n ,
\]
it follows that
\[
p_j(x)-p_i(x)= \int_{x_i}^{x_j}S'(u)du+\int_{x_j}^xp'_j(u)du-\int_{x_i}^xp'_i(u)du,
\]
and hence,
\begin{eqnarray*}
\|p_j-p_i\|_{I_i} &\le& \int_{x_i}^{x_j}|S'(u)|du+\int_{I_{i,j}}|p'_j(u)|du+\int_{I_i}|p'_i(u)|du\\
& \le& 2h_{i,j}\|S'\|_{I_{i,j}}+h_{i,j}\|p'_j\|_{I_{i,j}} =: \sigma_1 +\sigma_2 .
\end{eqnarray*}
We first estimate $\sigma_2$.
If $v\in I_j$, then it follows by \ineq{bound} that
\[
|p'_j(v)|=|S'(v)|   \leq \frac{\phi(\rho_n(v))}{ \rho_n(v)}        \leq c\frac{\phi(h_j)}{h_j},
\]
and since $p_j$ is a polynomial of degree $\le k-1$, this, in turn, implies that
\be \label{AA}
\sigma_2 = h_{i,j}\|p'_j\|_{I_{i,j}} \leq c h_{i,j} \frac{\phi(h_j)}{h_j}\left(\frac{h_{i,j}}{h_j}\right)^{k-2}
\leq c \phi(h_j)\left(\frac{h_{i,j}}{h_j}\right)^k.
\ee

We now estimate $\sigma_1$. First, note that it follows from \ineq{rho1} (with $y:=x_j$ and any $u\in I_{i,j}$) that $h_j^2\le ch_{i,j}\rho_n(u)$. If
 $\rho_n(u)<h_j$,  this implies
\[
\frac{\phi(\rho_n(u))}{\rho_n(u)}\le c\frac{\phi(h_j)}{h_j^2}h_{i,j}\le c\frac{\phi(h_j)}{h_j^k}h_{i,j}^{k-1},\quad u\in I_{i,j}.
\]
If $\rho_n(u)\ge h_j$, then
$$
\frac{\phi(\rho_n(u))}{\rho_n(u)}\le\frac{\phi(h_j)}{h_j^k}\rho_n^{k-1}(u)\le\frac{\phi(h_j)}{h_j^k}h_{i,j}^{k-1},\quad u\in I_{i,j},
$$
and so using \ineq{bound} again we have
\[
\sigma_1 = 2 h_{i,j}\|S'\|_{I_{i,j}}\le 2 h_{i,j}\left\|\frac{\phi(\rho_n)}{\rho_n}\right\|_{I_{i,j}}
\le c\frac{\phi(h_j)}{h_j^k}h_{i,j}^{k} .
\]
Combining this with \ineq{AA}, we obtain
\[
\|p_j-p_i\|_{I_i}\le c\,\phi(h_j) \left( \frac{h_{i,j} }{h_j }\right)^k,
\]
and the proof  is complete.
\end{proof}

\sect{Monotone polynomial approximation of   piecewise polynomials with  ``small'' derivatives} \label{section6}

All constants $C$ in this section may depend on $k$ and $\alpha$.

\begin{lemma} \label{stepeleven}
Let $\alpha>0$, $k\in\N$
and $\phi\in\Phi^k$,
be given. If $S\in\Sigma_{k,n} \cap\Delta^{(1)}$    is such that
\be\label{approx2}
|S'(x)|\le
\frac{\phi(\rho_n(x))}{\rho_n(x)},\quad x\in [x_{n-1},x_1]\setminus
\{x_j\}_{j=1}^{n-1},
\ee
\be \label{jumps}
0\leq S(x_j+) - S(x_j-) \leq \phi(\rho_n(x_j)), \quad 1\leq j \leq n-1 ,
\ee
and
\be\label{approx3}
S'(x)=0, \quad x\in [-1, x_{n-1}) \cup (x_1, 1],
\ee
then there is a polynomial $P \in\Delta^{(1)}\cap\Poly_{Cn}$   such that
\be\label{approx1}
|S(x)-P (x)|\le C\delta_n^\alpha(x) \,\phi\left(\rho_n(x)\right),\quad x\in [-1,1].
\ee
\end{lemma}

Note that, clearly, condition \ineq{jumps} is automatically satisfied at all knots $x_j$ where   $S$ is continuous.

\begin{proof}
Let
\[
S_1(x):=\begin{cases} S(x_j),&\quad x\in[x_j,x_{j-1}),\quad 2\le j\le n,\\
S(x_1),&\quad x\in[x_1,1].
\end{cases}
\]
 Clearly, \ineq{approx2} and \ineq{jumps} imply
\be\label{approx4}
|S(x)-S_1(x)|\le c\phi(\rho_n(x)),\quad x\in[-1,1],
\ee
and \ineq{approx3} yields (recall that $S$ is right continuous)
\be\label{approx5}
S_1(x)=S(x),\quad x\in I_1\cup I_n.
\ee
We may write,
\begin{eqnarray*}
S_1(x) &=& \sum_{j=2}^nS(x_j)\bigl(\chi_j(x)-\chi_{j-1}(x)\bigr) +S(x_1)\chi_1(x) \\
& = & S(-1)+\sum_{j=1}^{n-1}\bigl(S(x_j)-S(x_{j+1})\bigr)\chi_j(x)  ,\quad x\in[-1,1] .
\end{eqnarray*}
Let
\[
P(x):=S(-1)+\sum_{j=1}^{n-1}\bigl(S(x_j)-S(x_{j+1})\bigr)\tau_j(x)    ,\quad x\in[-1,1],
\]
where $\tau_j$ are the polynomials from   \lem{lem5} with the same $\alpha$ and $\beta=k+2$.

Then, $P$  is a nondecreasing polynomial of degree $\le Cn$ and, in view of \ineq{approx4} and \ineq{approx5}, we only need to estimate $|S_1(x)-P(x)|$.
First, note that \ineq{hjrho} implies, for all $1\leq j\leq n$ and $x\in [-1,1]$,
\[
\phi(h_j) \leq \phi\left( c\psi_j^{-1}(x)\rho_n(x) \right) \leq C\,\psi_j^{-k} \phi\left(  \rho_n(x) \right).
\]
Now, since \ineq{approx2}  and \ineq{jumps} imply that
\[
|S(x_j)-S(x_{j+1})|\le C\phi(h_j),\quad 1\le j\le n-1,
\]
using \ineq{tauj}, we conclude that,  for all $1\leq j\leq n-1$ and $x\in [-1,1]$,
\begin{eqnarray*}
|S(x_j)-S(x_{j+1})||\chi_j(x)-\tau_j(x)| &\le&  C \phi(h_j)\delta_n^{\alpha}(x)\psi^{k+2}_j(x) \\
& \leq &
C \phi \left(  \rho_n(x) \right) \delta_n^{\alpha}(x)\psi_j^{2}(x) .
\end{eqnarray*}
Therefore, by \ineq{sumpsi}, we have
\begin{eqnarray*}
|S_1(x)-P(x)|&\le & \sum_{j=1}^{n-1}|S(x_j)-S(x_{j+1})||\chi_j(x)-\tau_j(x)|\\
& \leq &
C \phi \left(  \rho_n(x) \right)   \delta_n^{\alpha}(x)  \sum_{j=1}^{n-1}\psi_j^{2}(x) \\
& \leq &
C  \phi \left(  \rho_n(x) \right)   \delta_n^{\alpha}(x) .
\end{eqnarray*}
Combined with \ineq{approx4} and \ineq{approx5}, our proof is complete.
\end{proof}

\sect{On one partition of unity} \label{section7}

\begin{lemma}\label{lem4.4}
Let $\alpha_1,\beta_1>0$, and let $n,n_1\in\N$, $n_1 > n$,  be such that $n_1$ is divisible by $n$.
Then, there is a collection $\{\widetilde T_{j,n_1}\}_{j=1}^n$ of polynomials $\widetilde T_{j,n_1}\in\Poly_{C(\alpha_1,\beta_1) n_1}$, such that
the following relations hold:
\be\label{7.19}
\sum_{j=1}^n\widetilde T_{j,n_1}(x)\equiv1,\quad x\in[-1,1],
\ee
\be\label{7.20}
\widetilde T_{1,n_1}'(x)\ge0\quad{\rm and}\quad\widetilde T_{n,n_1}'(x)\le0,\quad x\in[-1,1],
\ee
\be\label{7.21}
|\widetilde T_{j,n_1}(x)|\le C
\delta_{n_1}^{\alpha_1}(x)
\left(\frac{\rho_{n_1}(x)}{\rho_{n_1}(x)+\dist (x,I_j)}\right)^{\beta_1},
\ee
for all
\[
x\in \DD_j := \begin{cases} [-1,x_1],&\quad\text{if}\quad j=1,\\
[x_{n-1},1],&\quad\text{if}\quad j=n,\\
[-1,1],&\quad\text{if}\quad 2\leq j \leq n-1,\\
\end{cases}
\]
and
\begin{eqnarray}  \label{7.21q}
|\widetilde T_{j,n_1}^{(q)}(x)| &\le &  C \frac{\delta_{n_1}^{\alpha_1}(x)}{\rho^q_{n_1}(x)}
\left(\frac{\rho_{n_1}(x)}{\rho_{n_1}(x)+\dist (x,I_j)}\right)^{\beta_1},  \\ \nonumber
&&  \qquad  1\leq q\leq \alpha_1 \;\; \text{and} \;\;  x\in [-1,1] ,
\end{eqnarray}
where all constants $C$ depend only on $\alpha_1$, $\beta_1$ and are independent of the ratio $n_1/n$.
\end{lemma}

\begin{proof} Let $\tau_{i,n_1}$, $1\le i\le n_1-1$, be the polynomials from   \lem{lem5}  with
$\alpha$ and $\beta$ to be prescribed,
and denote $\tau_{0,n_1}\equiv0$ and $\tau_{n_1,n_1}\equiv1$.

Set
\[
T_{i,n_1}:=\tau_{i,n_1}-\tau_{i-1,n_1},\quad 1\leq i \leq n_1,
\]
and note that
\be\label{unit}
\sum_{i=1}^{n_1}T_{i,n_1}\equiv 1.
\ee
Let $d:= n_1/n$ and define
\be\label{tildeT}
\widetilde T_{j,n_1}:=\sum_{i=d(j-1)+1}^{dj}T_{i,n_1}=\sum_{I_{i,n_1}\subset I_j}T_{i,n_1},\quad 1\leq j \leq n.
\ee
Then \ineq{7.19} readily follows by \ineq{unit}, and \ineq{7.20} is evident.

We now note that, for all  $x\in [-1,1]$,
\[
\psi_{i \pm 1, n_1}(x) < 4 \psi_{i, n_1}(x) , \quad 1\leq i \leq n_1,
\]
(recall that $\psi_{0,n_1}(x) \equiv 0$ and $\psi_{n_1+1, n_1}(x) \equiv 0$)
and
\begin{align*}
\chi_{i,n_1}(x)-\chi_{i-1, n_1}(x)&=\chi_{[x_{i,n_1},x_{i-1, n_1})}(x) =\delta_{n_1}^\alpha (x)\chi_{[x_{i, n_1},x_{i-1, n_1})}(x)\\&\leq
C\delta_{n_1}^\alpha(x) \psi_{i, n_1}^{\beta}(x),
\end{align*}
for $2\leq i \leq n_1-1$.

Hence, by \ineq{tauj},
\begin{eqnarray*}
\lefteqn{|T_{i,n_1}(x)|}\\
&\leq & |\tau_{i,n_1}(x)-\chi_{i, n_1}(x)| + |\chi_{i, n_1}(x) - \chi_{i-1, n_1}(x)| +   |\chi_{i-1, n_1}(x)  -\tau_{i-1,n_1}(x)| \\
&\leq & C \delta_{n_1}^\alpha(x) \psi_{i, n_1}^{\beta}(x), \quad 2\leq i \leq n_1-1, \quad x\in [-1,1] .
\end{eqnarray*}
If $i=1$, then, for $x\in [-1,x_{1, n_1}) \supset [-1, x_{1, n}]$,
\[
|T_{1,n_1}(x)| = |\tau_{1,n_1}(x)| = |\tau_{1,n_1}(x) - \chi_{1, n_1}(x)| \leq C \delta_{n_1}^\alpha(x) \psi_{i, n_1}^{\beta}(x) ,
\]
and similarly, for $i=n_1$ and $x\in [x_{n_1-1, n_1}, 1]  \supset  [x_{n_1-1, n}, 1]$,
\[
|T_{n_1,n_1}(x)| = |1- \tau_{n_1-1,n_1}(x)| = |\chi_{n_1-1, n_1}(x)- \tau_{n_1-1,n_1}(x)| \leq C \delta_{n_1}^\alpha(x) \psi_{i, n_1}^{\beta}(x) .
\]
Hence, for $x \in \DD_j$,
\be \label{funct}
|\widetilde T_{j,n_1}(x)| \le  C \delta_{n_1}^\alpha(x)\sum_{I_{i,n_1}\subset I_j}  \psi_{i, n_1}^{\beta}(x) .
\ee
Similarly, it follows from \ineq{derivatives} that, for all $x\in [-1,1]$,
\be \label{deriv}
|\widetilde T_{j,n_1}^{(q)}(x)|\leq C \delta_{n_1}^\alpha(x) \sum_{I_{i,n_1}\subset I_j} h_{i, n_1}^{-q}\psi_{i,n_1}^{\beta}(x) , \quad 1\leq q \leq \alpha.
\ee
Therefore, we may treat \ineq{funct} as a particular case of \ineq{deriv} for $q=0$, keeping in mind that $x$ is assumed to be in $\DD_j$ in that case.

We are now ready to prove \ineq{7.21} and \ineq{7.21q}.
First, we note that  \ineq{rho} and \ineq{rho1} imply that
\[
\psi_{i, n_1}^2(x) = \left(\frac{h_{i,n_1}}{|x-x_{i,n_1}|+h_{i,n_1}}\right)^2\le c \frac{\rho_{n_1}(x)}{|x-x_{i,n_1}|+\rho_{n_1}(x)} .
\]
Hence,
\begin{eqnarray*}
 |\widetilde T_{j,n_1}^{(q)}(x)|
&\le& C \delta_{n_1}^\alpha(x)\sum_{I_{i,n_1}\subset I_j}  h_{i, n_1}^{-q}
\left( \frac{h_{i,n_1}}{|x-x_{i,n_1}|+\rho_{n_1}(x)}\right)^{q+1}   \times \\
&& \times  \left(\frac{\rho_{n_1}(x)}{|x-x_{i,n_1}|+\rho_{n_1}(x)}\right)^{(\beta-q-1)/2}\\
&=& C \delta_{n_1}^\alpha(x)\rho_{n_1}^{(\beta-q-1)/2}(x)\sum_{I_{i,n_1}\subset I_j}\frac{h_{i,n_1}}{(|x-x_{i,n_1}|+\rho_{n_1}(x))^{(\beta+q+1)/2}}\\
&\le&  C \delta_{n_1}^\alpha(x)\rho_{n_1}^{(\beta-q-1)/2}(x) \int_{\dist(x, I_j)}^\infty\frac{du}{(u+\rho_{n_1}(x))^{(\beta+q+1)/2}}\\
&=&
C \frac{\delta_{n_1}^\alpha(x)}{\rho_{n_1}^{q}(x)}     \left(\frac{\rho_{n_1}(x)}{\dist(x,I_j)+\rho_{n_1}(x)}\right)^{(\beta-q-1)/2}.
\end{eqnarray*}
It remains to set $\alpha  := \lceil\alpha_1 \rceil$ and $\beta := 2\beta_1+\alpha_1+1$, and the proof is complete.
\end{proof}

In the proof below, we need estimates \ineq{7.21} and \ineq{7.21q} for $x\in\DD_j$ with $\delta_{n_1}(x)$ replaced by $\delta_{n}(x)$ which is smaller near the endpoints of $[-1,1]$.

\begin{corollary}\label{coroflem4.4}
Let $\alpha_2,\beta_2>0$, and let $n,n_1\in\N$, $n_1 > n$,  be such that $n_1$ is divisible by $n$.
Then, there is a collection $\{\widetilde T_{j,n_1}\}_{j=1}^n$ of polynomials $\widetilde T_{j,n_1}\in\Poly_{C(\alpha_2,\beta_2) n_1}$, such that
\ineq{7.19} and \ineq{7.20} are valid, and
\be\label{cor7.21q}
|\widetilde T_{j,n_1}^{(q)}(x)|\le  C \frac{\delta_{n}^{\alpha_2}(x)}{\rho^q_{n_1}(x)}
\left(\frac{\rho_{n_1}(x)}{\rho_{n_1}(x)+\dist (x,I_j)}\right)^{\beta_2}, \quad 0\leq q\leq \alpha_2,
\ee
for all $x\in \DD_j$,
where all constants $C$ depend only on $\alpha_2$, $\beta_2$ and are independent of the ratio $n_1/n$.
\end{corollary}

\begin{proof}
Since
\[
\frac{\delta_{n_1}(x)}{\delta_{n}(x)} \leq
\begin{cases} \ds
\frac{n_1}{n}, &  \mbox{\rm if}\quad \varphi(x) \leq 1/n, \\
1, &  \mbox{\rm if}\quad \varphi(x) >  1/n,
\end{cases}
\]
we only need to prove \ineq{cor7.21q} for $x\in \widetilde \DD_j := \DD_j \cap \left\{ x \st \varphi(x) \leq 1/n \right\}$ (for all other $x\in \DD_j$ it is an immediate corollary of \ineq{7.21} and \ineq{7.21q} with $\alpha_1  =\alpha_2$ and $\beta_1  =\beta_2$).
Note that for all $1\leq j\leq n$ and $x\in \widetilde \DD_j$,
\[
\dist(x, I_j) \geq \dist(\widetilde \DD_j, I_j) \geq \dist(x_1, \sqrt{1-n^{-2}}) \geq n^{-2} .
\]
Hence, it follows from \ineq{7.21} and \ineq{7.21q}, that for all $0\leq q\leq \alpha_1$, $1\leq j\leq n$ and $x\in \widetilde \DD_j$, we have
\begin{eqnarray*}
|\widetilde T_{j,n_1}^{(q)}(x)| &\le&
  C \frac{\delta_{n}^{\alpha_1}(x)}{\rho^q_{n_1}(x)}
  \left( \frac{n_1}{n} \right)^{\alpha_1}
\left(\frac{\rho_{n_1}(x)}{\rho_{n_1}(x)+\dist (x,I_j)}\right)^{\beta_1} \\
& \le&
 C \frac{\delta_{n}^{\alpha_1}(x)}{\rho^q_{n_1}(x)}
  \left( \frac{n_1}{n} \right)^{\alpha_1}
 \left(\frac{\rho_{n_1}(x)}{\rho_{n_1}(x)+n^{-2}}\right)^{\alpha_1}
\left(\frac{\rho_{n_1}(x)}{\rho_{n_1}(x)+\dist (x,I_j)}\right)^{\beta_1-\alpha_1} \\
& \le &
C \frac{\delta_{n}^{\alpha_1}(x)}{\rho^q_{n_1}(x)}
\left(\frac{\rho_{n_1}(x)}{\rho_{n_1}(x)+\dist (x,I_j)}\right)^{\beta_1-\alpha_1} ,
\end{eqnarray*}
since
\[
 \frac{n_1 \rho_{n_1}(x)}{n \rho_{n_1}(x)+1/n} \leq \frac{\varphi(x) +  1/n_1}{ n\varphi(x)/n_1 + 1/n} \leq 1 ,
\]
for $n_1\geq n$ and $\varphi(x)\leq 1/n$.
It remains to set $\alpha_1 :=\alpha_2$ and $\beta_1 := \alpha_2+\beta_2$.
\end{proof}

\sect{Simultaneous polynomial approximation of piecewise polynomials and their derivatives} \label{section8}

All constants $C$ in this section depend on $k$ and $\gamma$.

  We need the following result which is similar to \cite{LS2002}*{Lemma 18} and which is proved in a similar way.

\begin{lemma}\label{uncon}
Let $\gamma >0$, $k\in\N$, $\phi\in\Phi^k$, and let $n, n_1\in\N$ be such that $n_1$ is divisible by $n$. If
 $S\in\Sigma_{k,n}$, then there exists a polynomial $D_{n_1}(\cdot, S)$ of degree $\leq Cn_1$ such that
\be \label{7.23'}
\left|S(x)-D_{n_1}(x, S)\right|\le
C\delta_n^{\gamma}(x)  \phi(\rho_n(x))   b_k(S,\phi).
\ee
Moreover, if   $S\in \C^{r-1}[-1,1]$ for some $r\in\N$, $r\leq k$, and $A := [x_{\mu^*}, x_{\mu_*}]$, $0\leq \mu_* < \mu^* \leq n$, then for all $x\in A\setminus \{x_j\}_{j=1}^{n-1}$ and $0\leq q \leq r$,
 we have
 \begin{eqnarray} \label{7.24'}
  \left|S^{(q)}(x)-D_{n_1}^{(q)}(x,S)\right|
 &\le &  C \delta_n^{\gamma}(x)
\frac { \phi(\rho_n(x))}{\rho_n^q(x)} \bigg(  b_k(S,\phi,A) +   b_k(S,\phi) \times     \\ \nonumber
 &&
\times   \frac{n}{n_1} \left(\frac{\rho_n(x)}{ \dist(x,[-1,1] \setminus A)}\right)^{\gamma+1}   \bigg).
\end{eqnarray}
The constants $C$ above are independent of the ratio $n_1/n$.
\end{lemma}

\begin{proof}
We denote
\be\label{7.22}
D_{n_1}(x,S):=\sum_{j=1}^np_j(x)\widetilde T_{j,n_1}(x),
\ee
where $\widetilde T_{j,n_1}$ are polynomials of degree $\leq C(\alpha_2,\beta_2) n_1$ from the statement of \cor{coroflem4.4}. Note that $D_{n_1}(\cdot,S)$ is a polynomial of degree $< k+ C(\alpha_2,\beta_2) n_1$.
The parameters $\alpha_2$ and $\beta_2$ depend on $\gamma$ and $k$ are chosen to  be sufficiently large. For example, $\alpha_2 = \gamma$ and $\beta_2 = \gamma + 4k+5$ will do.

For the sake of brevity,  we will use the notation $\rho := \rho_n(x)$, $\delta := \delta_n(x)$,  $\rho_1 := \rho_{n_1}(x)$ and $\widetilde T_j := \widetilde T_{j,n_1}$.
Recall that $I_{i,j}$ is the smallest interval  containing both $I_{i}$ and $I_j$, and  $h_{i,j}:=|I_{i,j}|$.
Suppose now that $x$ is fixed and let $1\leq \nu \leq n$ be the smallest number such that $x\in I_\nu$ (\ie if $x=x_\eta$, then $x$ belongs to both $I_\eta$ and $I_{\eta+1}$, and  we pick $\nu=\eta$).

We now observe that \ineq{anotherauxest} and \ineq{distance1}  imply
\be \label{hnuj}
\frac{h_\nu}{h_j} < 5 \frac{\rho}{h_j} <  40 \frac{|x-x_j|+\rho}{\rho} \sim \frac{\rho+ \dist(x, I_j)}{\rho} , \quad 1\leq j \leq n .
\ee
Also,
\be\label{auxinequality1}
\frac{h_{\nu,j} }{h_\nu }  \leq c \frac{\rho+ \dist(x, I_j)}{\rho} ,  \quad 1\leq j \leq n .
\ee
Indeed,  if $|j-\nu|\leq 1$, then it is enough to note that \ineq{rho} implies that $h_{\nu, j} \sim h_\nu$.  If $|j-\nu|\geq 2$, then we use the fact that there is  at least one interval $I_i$ between $I_\nu$ and $I_j$, and so \ineq{rho} implies
\[
h_{\nu, j} = h_\nu + h_j + \dist(I_\nu, I_j) \leq  h_\nu +     4 \dist(I_\nu, I_j)
\leq h_\nu + 4 \dist(x, I_j),
\]
and \ineq{auxinequality1} follows.

Since $S(x)=p_\nu(x)$,  \ineq{7.19}    implies
\begin{eqnarray*}
S(x)-D_{n_1}(x, S)&=& S(x)\sum_{j=1}^n\widetilde T_j(x)-\sum_{j=1}^np_j(x)\widetilde T_j(x) \\
&=&
\sum_{1\leq j \leq n, j\ne \nu} (p_\nu(x)-p_j(x))\widetilde T_j(x) ,
\end{eqnarray*}
and so
\begin{eqnarray*}
  S^{(q)}(x)-D_{n_1}^{(q)}(x, S) & = & \sum_{1\leq j \leq n, j\ne \nu}        \left( \left(p_\nu (x)-p_j (x)\right)  \widetilde T_j (x) \right)^{(q)} \\
&=&
\sum_{1\leq j \leq n, j\ne \nu}  \sum_{i=0}^q  \binom{q}{i}      \left(p_\nu^{(i)}(x)-p_j^{(i)}(x)\right)  \widetilde T_j^{(q-i)}(x)  ,
\end{eqnarray*}
with the assumption that $x\not\in \left\{ x_j\right\}_{j=1}^{n-1}$ if $q\geq 1$, since $S^{(q)}$ may not exist at those points.
Note also that
$x\in\DD_j$ for all $1\leq j \leq n$, $j\ne \nu$, and so \ineq{cor7.21q} can be used for all polynomials $\widetilde T_j$ appearing in the above sum.

Now,  since
 \[
 \phi(h_j) \leq \phi(h_{\nu,j}) \leq \phi(h_\nu) \left( \frac{h_{\nu,j}}{h_\nu} \right)^k \leq c \phi(\rho) \left( \frac{h_{\nu,j}}{h_\nu} \right)^k ,
 \]
it follows from \ineq{bij}, \ineq{hnuj} and \ineq{auxinequality1} that,
for all $i\geq 0$ (of course, the inequality is trivial if $ i \ge k$),
\begin{eqnarray}\label{8.3}
\lefteqn{\mbox{\hspace{5mm}}  \|p_\nu^{(i)}-p_j^{(i)}\|_{I_\nu}  \le   c h_\nu^{-i} \|p_\nu -p_j \|_{I_\nu}
 \leq
c b_{\nu, j}(S,\phi) \frac{ \phi(h_j)}{h_\nu^i} \left(\frac{h_{\nu,j}}{h_j}\right)^k } \\ \nonumber
& \quad&  \leq
 c b_{\nu, j}(S,\phi) \frac{\phi(\rho)}{\rho^i} \left(\frac{h_{\nu,j}^2}{h_\nu h_j}\right)^k
  \leq
 c b_{\nu, j}(S,\phi) \frac{\phi(\rho)}{\rho^i} \left(\frac{\rho+ \dist(x, I_j)}{\rho}\right)^{3k} .
\end{eqnarray}
Observing that
\be \label{rho1rho}
\frac{\rho_1}{\rho_1+\dist (x,I_j)} \leq \frac{\rho}{\rho+\dist (x,I_j)}
\ee
and using \ineq{cor7.21q}
 we now conclude that, for all $0\leq i \leq q$ and $1\leq j \leq n$, $j\neq \nu$,
\[
  \left| \left(p_\nu^{(i)}(x)-p_j^{(i)}(x)\right)  \widetilde T_j^{(q-i)}(x) \right| \leq
C b_{\nu, j}(S,\phi) \delta^{\alpha_2}  \frac{\phi(\rho)}{\rho^i \rho_1^{q-i}}
\left(\frac{\rho_1}{\rho_1+\dist (x,I_j)}\right)^{\beta_2-3k}.
\]
If $i=q$, this becomes
\begin{align}  \label{ieqq}
\lefteqn{  \left| \left(p_\nu^{(q)}(x)-p_j^{(q)}(x) \right)  \widetilde T_j (x) \right|  }\\ \nonumber
& \quad \leq
 C b_{\nu, j}(S,\phi) \delta^{\alpha_2}  \frac{\phi(\rho)}{\rho^q}
\left(\frac{\rho_1}{\rho_1+\dist (x,I_j)}\right)^{\beta_2-3k},
\end{align}
and, in particular,
if $i=q=0$, then
\be \label{qeq0}
 \left| \left(p_\nu (x)-p_j (x)\right)  \widetilde T_j (x) \right| \leq C b_{\nu, j}(S,\phi) \delta^{\alpha_2}  \phi(\rho) \left(\frac{\rho_1}{\rho_1+\dist (x,I_j)}\right)^{\beta_2-3k} .
\ee
Now, with  an additional assumption that $j\neq \nu\pm 1$ (which implies that  $\dist(x, I_j) > \rho/3 $), and using
 $\rho_1 /\rho \leq n/n_1$, we have
\begin{eqnarray} \label{jnenu}
\lefteqn{ \left| \left(p_\nu^{(i)}(x)-p_j^{(i)}(x)\right)  \widetilde T_j^{(q-i)}(x) \right| }\\ \nonumber
& \leq  &
C  b_{\nu, j}(S,\phi) \delta^{\alpha_2} \frac{\phi(\rho)}{\rho^q}     \frac{\rho_1}{\rho}       \left( \frac{\rho }{ \rho_1+\dist (x,I_j)  } \right)^{q-i+1}
\times \\ \nonumber
&& \times
\left(\frac{\rho_1}{\rho_1+\dist (x,I_j)}\right)^{\beta_2-3k-q+i-1} \\ \nonumber
 &\leq&
C b_{\nu, j}(S,\phi) \delta^{\alpha_2} \frac{\phi(\rho)}{\rho^q}     \frac{n}{n_1}
\left(\frac{\rho_1}{\rho_1+\dist (x,I_j)}\right)^{\beta_2-3k- q-1} .
\end{eqnarray}

It remains to consider the case $q\geq 1$, $i\leq q- 1$ and $j=\nu \pm 1$. We only consider the case $j=\nu+1$, the case $j=\nu-1$ being completely analogous.

We now have to use the fact that $S$ is assumed to be sufficiently smooth. Indeed, if
  $S\in\C^{q-1}[-1,1]$, we have $p_\nu^{(l)}(x_\nu) = p_{\nu+1}^{(l)}(x_\nu)$, $0\leq l \leq q-1$, and so by \ineq{8.3},
\begin{eqnarray*}
\left|p_\nu^{(i)}(x)- p_{\nu+1}^{(i)}(x) \right| & = & \frac{1}{(q-i-1)!} \left| \int_{x_\nu}^x (x-u)^{q-i-1} \left( p_\nu^{(q)}(u)- p_{\nu+1}^{(q)}(u) \right) du \right| \\
&\leq & |x-x_\nu|^{q-i} \norm{p_\nu^{(q)} - p_{\nu+1}^{(q)}}{I_\nu} \\
& \leq &
c  |x-x_\nu|^{q-i}   b_{\nu, \nu+1}(S,\phi) \frac{\phi(\rho)}{\rho^{q}} \left(\frac{\rho+ |x-x_\nu|}{\rho}\right)^{3k} .
\end{eqnarray*}
Therefore,
\begin{eqnarray*}
\lefteqn{ \left| \left(p_\nu^{(i)}(x)-p_{\nu+1}^{(i)}(x)\right)  \widetilde T_{\nu+1}^{(q-i)}(x) \right| }\\ \nonumber
& \leq &
C b_{\nu, \nu+1}(S,\phi) \delta^{\alpha_2}  \frac{\phi(\rho) |x-x_\nu|^{q-i}}{\rho^{q} \rho_1^{q-i}}
\left(\frac{\rho_1}{\rho_1+|x-x_\nu|}\right)^{\beta_2-3k}\\ \nonumber
& \leq &
C b_{\nu, \nu+1}(S,\phi) \delta^{\alpha_2}  \frac{\phi(\rho) }{\rho^{q} }
\left(\frac{\rho_1}{\rho_1+|x-x_\nu|}\right)^{\beta_2-3k-q+i} .
\end{eqnarray*}
In summary,   the   estimate
\begin{eqnarray} \label{close}
\lefteqn{ \left| \left(p_\nu^{(i)}(x)-p_{\nu\pm 1}^{(i)}(x)\right)\widetilde T_{\nu\pm 1}^{(q-i)}(x) \right| }   \\ \nonumber
& \quad &  \leq
C b_{\nu, \nu\pm 1}(S,\phi) \delta^{\alpha_2}  \frac{\phi(\rho) }{\rho^{q} } \left(\frac{\rho_1}{\rho_1+\dist (x,I_{\nu\pm 1})}\right)^{\beta_2-3k-q},
\end{eqnarray}
is valid for all $0\leq i \leq q$ provided that $S\in\C^{q-1}[-1,1]$ (for $i=q$ it follows from \ineq{ieqq}).

Using \ineq{qeq0}, \ineq{rho1rho}, \ineq{sumest} and  the estimate $b_{\nu, j}(S,\phi) \leq b_k(S,\phi)$, we have
\begin{eqnarray} \label{firstestimate}
 \lefteqn{  \left| S(x)-D_{n_1}(x, S) \right| }\\ \nonumber
 & \leq &
C b_k(S,\phi)  \delta^{\alpha_2}  \phi(\rho)  \sum_{1\leq j \leq n, j\ne \nu} \left(\frac{\rho}{\rho +\dist (x,I_j)}\right)^{\beta_2-3k} \\ \nonumber
& \leq &
C b_k(S,\phi)  \delta^{\gamma} \phi(\rho) ,
\end{eqnarray}
and \ineq{7.23'} is proved.

We will now prove (\ref{7.24'}). Suppose that $S\in\C^{r-1}[-1,1]$ and $0\leq q\leq r$.
We write
\begin{eqnarray*}
 \lefteqn{ S^{(q)}(x)-D_{n_1}^{(q)}(x, S)}\\
& =  &
\sum_{1\leq j \leq n, j\ne \nu}        \left( \left(p_\nu (x)-p_j (x)\right)  \widetilde T_j (x) \right)^{(q)}  \\
& = &
\left(  \sum_{j \in \J_1} +  \sum_{j \in \J_2} +  \sum_{j \in \J_3} +  \sum_{j \in \J_4}  \right) \left( \left(p_\nu (x)-p_j (x)\right)  \widetilde T_j (x) \right)^{(q)} \\
& =: &
\sigma_1(x) + \sigma_2(x) + \sigma_3(x) + \sigma_4(x) ,
\end{eqnarray*}
where
\begin{eqnarray*}
\J_1 &:=&  \left\{ j \st 1\leq j \leq n, I_j \subset A , j \ne \nu, \nu \pm 1  \right\}, \\
\J_2 &:=&  \left\{ j \st 1\leq j \leq n, I_j \not\subset A , j \ne \nu, \nu \pm 1  \right\} ,\\
\J_3 &:=&  \left\{ j \st  1\leq j \leq n,  j=\nu+1  \right\} , \\
\J_4 &:=&  \left\{ j \st  1\leq j \leq n, j = \nu - 1    \right\} .
\end{eqnarray*}
Note that some of the  sets $\J_l$ may be empty (making the corresponding functions $\sigma_l \equiv 0$). For example, if $\nu=1$, then $\J_4 = \emptyset$ and $\sigma_4\equiv 0$;  if $A \subset I_{\nu+1} \cup I_\nu \cup I_{\nu+1} $, then $\J_1 = \emptyset$ and $\sigma_1\equiv 0$,  etc.

In order to estimate $\sigma_1$, using \ineq{jnenu}, \ineq{rho1rho},  \ineq{sumest} and the estimate $b_{\nu, j}(S,\phi) \leq b_k(S,\phi, A)$, $j\in \J_1$, we have
\begin{eqnarray*}
| \sigma_1(x) | &\leq &
 C b_k(S,\phi, A) \delta^{\alpha_2} \frac{\phi(\rho)}{\rho^q}     \frac{n}{n_1}
\sum_{j \in \J_1}
\left(\frac{\rho_1}{\rho_1+\dist (x,I_j)}\right)^{\beta_2-3k-q-1}  \\
& \leq &
C b_k(S,\phi, A) \delta^{\gamma} \frac{\phi(\rho)}{\rho^q} .
\end{eqnarray*}

To estimate $\sigma_2$, we  use \ineq{jnenu}, \ineq{rho1rho}, \ineq{distance1}, \ineq{anotherauxest}  and    $b_{\nu, j}(S,\phi) \leq b_k(S,\phi)$, $j\in \J_2$, and write
\begin{eqnarray*}
| \sigma_2(x) | &\leq &
 C b_k(S,\phi) \delta^{\alpha_2} \frac{\phi(\rho)}{\rho^q}    \frac{n}{n_1}
\sum_{j \in \J_2}
\left(\frac{\rho_1}{\rho_1+\dist (x,I_j)}\right)^{\beta_2-3k- q-1}  \\
 &\leq &
 C b_k(S,\phi) \delta^{\alpha_2} \frac{\phi(\rho)}{\rho^q}    \frac{n}{n_1}
\sum_{j \in \J_2} \frac{h_j}{\rho}
\left(\frac{\rho }{\rho +|x-x_j|}\right)^{\beta_2-3k-q-2}  \\
&\leq &
 C b_k(S,\phi) \delta^{\gamma} \frac{\phi(\rho)}{\rho^q}    \frac{n}{n_1}
\sum_{j \in \J_2} \frac{h_j}{\rho}
\left(\frac{\rho }{\rho +|x-x_j|}\right)^{\gamma+2}  \\
&\leq &
 C b_k(S,\phi) \delta^{\gamma} \frac{\phi(\rho)}{\rho^q}    \frac{n}{n_1} \rho^{\gamma+1}
\sum_{j \in \J_2}
 \frac{h_j }{(\rho +|x-x_j|)^{\gamma+2}}  \\
&\leq &
 C b_k(S,\phi) \delta^{\gamma} \frac{\phi(\rho)}{\rho^q}    \frac{n}{n_1} \rho^{\gamma+1}
\int_{\dist(x,[-1,1]\setminus A)}^\infty
 \frac{du }{(\rho +u)^{\gamma+2}}  \\
 & \leq &
 C b_k(S,\phi) \delta^{\gamma} \frac{\phi(\rho)}{\rho^q}    \frac{n}{n_1}
 \left(\frac{\rho}{\rho+\dist(x,[-1,1]\setminus A)}\right)^{\gamma+1} .
\end{eqnarray*}
Finally, we will estimate $\sigma_3$ (the proof for $\sigma_4$ is completely analogous).
First, if $I_{\nu+1} \subset A$, then $b_{\nu, \nu+1}(S,\phi) \leq b_k(S,\phi, A)$ and so \ineq{close} yields
\[
|\sigma_3(x)| \leq C b_k(S,\phi, A) \delta^{\gamma} \frac{\phi(\rho)}{\rho^q} .
\]
If $I_{\nu+1} \not\subset A$, then $\nu=\mu^*$ (and so $\dist(x,[-1,1]\setminus A)  \leq |x-x_\nu| = \dist(x, I_{\nu+1})$),
$b_{\nu, \nu+1}(S,\phi) \leq b_k(S,\phi)$,
and again using \ineq{close}, we have
\begin{eqnarray*}
|\sigma_3(x)| &\leq&  C b_k(S,\phi) \delta^{\gamma} \frac{\phi(\rho)}{\rho^q}\left(\frac{\rho_1}{\rho_1+\dist(x, I_{\nu+1})}\right)^{\gamma+1} \\
&\leq &
C b_k(S,\phi) \delta^{\gamma} \frac{\phi(\rho)}{\rho^q}   \frac{\rho_1}{\rho}   \left(\frac{\rho }{\rho_1+\dist(x,[-1,1]\setminus A)}\right)^{\gamma+1} \\
&\leq &
C b_k(S,\phi) \delta^{\gamma} \frac{\phi(\rho)}{\rho^q}   \frac{n}{n_1}   \left(\frac{\rho }{\dist(x,[-1,1]\setminus A)}\right)^{\gamma+1} .
\end{eqnarray*}
The proof is now complete.
\end{proof}

\sect{One particular polynomial with controlled first derivative} \label{yetanother}
All constants $C$ in this section depend on $\alpha$ and $\beta$.

 The following lemma is a modification of  \cite{LS2002}*{Lemma 10}.

\begin{lemma}\label{QM}
Let $\alpha, \beta >0$, $k\in\N$ and $\phi\in\Phi^k$. Also,
let $E\subset [-1,1]$ be a closed interval which is the union of $m_E \ge 100$ of the
intervals $I_j$, and let a set $J\subset  E$ consist of $m_J$   intervals $I_j$, where
$1\le m_J < m_E/4$. Then there exists a
polynomial $Q_n(x)=Q_n(x,E,J)$ of degree $\le Cn$, satisfying
\begin{align}  \label{Qu}
Q_n'(x)&\ge C
\frac {m_E}{m_J} \delta_n^{8\alpha}(x) \frac{\phi(\rho_n(x))}{\rho_n(x)}
\left(\frac{\rho_n(x)}{\max\{\rho_n(x),\dist (x,E)\}}\right)^{60(\alpha+\beta)+4k+2}, \\ \nonumber
 &  \quad x\in J\cup([-1,1]\setminus E),
\end{align}
\be\label{Qu1}
Q_n'(x)\ge -\delta_n^{\alpha}(x) \frac{\phi(\rho_n(x))}{\rho_n(x)},\quad x\in E\setminus J,
\ee
and
\begin{eqnarray}  \label{Qu2}
|Q_n(x)| &\le &  C \,m_E^{k+3}\delta_n^{\alpha}(x) \rho_n(x)\,\phi(\rho_n(x))\sum_{j: \, I_j\subset  E}\frac{h_j}{(|x-x_j|+\rho_n(x))^2}, \\ \nonumber
&& \qquad   x\in[-1,1].
\end{eqnarray}
\end{lemma}

\begin{proof} First, it will be shown that we may assume that $I_n \not\subset E$ provided that   the condition $m_J < m_E/4$ is replaced by a slightly weaker  $m_J \leq m_E/4$.

Suppose that the lemma is proved for all $E_1$ such that $I_n \not\subset E_1$,  let   $E$ be such that $I_n\subset E$,   set $E_1 :=\left(E\setminus I_n\right)^{\cl}$ and $Q_n(x, E, J) := Q_n(x, E_1, J_1)$ (with $J_1$ to be prescribed), and consider the following three cases noting that, if the inequality in \ineq{Qu} holds for a particular $x$ then the the inequality in \ineq{Qu1} holds for that $x$ as well, and that $\max\{\rho_n(x),\dist (x,E_1)\} \sim \max\{\rho_n(x),\dist (x,E)\}$.

{\bf Case (i):}  If $I_n \subset J$ and $m_J\ge 2$, then we  define   $J_1 :=\left(J\setminus I_n\right)^{\cl}$, and note that $E_1\setminus J_1 = E\setminus J$ (and so
$J_1\cup([-1,1]\setminus E_1) = J\cup([-1,1]\setminus E)$),
$1\leq m_{J_1} <  m_{E_1}/4$,
and $m_{E_1}/m_{J_1}   < 2 m_{E}/m_{J}$.

{\bf Case (ii):} If $m_J=1$ and $J=I_n$, then
we define $J_1 := I_{n-1}$, and note that $E_1\setminus J_1 \subset E\setminus J$ (and so  $J_1\cup([-1,1]\setminus E_1) \supset  J\cup([-1,1]\setminus E)$),
$1=m_{J_1}< m_{E_1}/4$, and $m_{E_1}/m_{J_1}   <   m_{E}/m_{J}$.

{\bf Case (iii):} If $I_n \not\subset J$, then $J\subset E_1$ and we define $J_1 :=J$. Then, $E_1\setminus J_1 \subset E\setminus J$,
$1\leq m_{J_1} \leq m_{E_1}/4$ (since $4m_J < m_E$ implies that $4m_J \leq m_E-1 = m_{E_1}$),
 and
$m_{E_1}/m_{J_1}   <   m_{E}/m_{J}$.

Hence, in the rest of the proof, we assume that $I_n\not\subset E$ and $m_J\le m_E/4$.

It is convenient to use the notation $\rho := \rho_n(x)$, $\delta := \delta_n(x)$ and $\psi_j := \psi_j(x)$.
It is also convenient to denote
\begin{eqnarray*}
\E &:=&  \left\{ 1\leq j \leq n \st I_j \subset E\right\} , \quad
\J  :=   \left\{ 1\leq j \leq n \st I_j \subset J\right\},\\
j_* &:=&  \min \left\{ j \st j\in \E\right\}, \quad  j^*  :=   \max \left\{ j \st j\in \E\right\} , \\
\A &:= & \J \cup \{j_*, j^*\} \andd \B := \E \setminus \A .
\end{eqnarray*}
Note that $j^* = j_* + m_E-1$, $E = [x_{j^*}, x_{j_*-1}]$,
$\#\E = m_E$, $m_J = \#\J \sim \#\A$, and $\#\B \sim m_E$.

Note that \ineq{hjrho} implies $c \psi_j^2  \rho \leq h_j \leq c \psi_j^{-1}  \rho$, and so
\be \label{upperphi}
\phi(h_j) \leq   \max\{1, h_j^k \rho^{-k}\} \phi(\rho) \leq c \psi_j^{-k} \phi(\rho).
\ee
Similarly,
\be \label{lowphi}
 \phi(h_j)   \geq   \min\{ 1, h_j^k \rho^{-k}\} \phi(\rho) \geq c \psi_j^{2k} \phi(\rho) .
\ee

Let
\[
Q_n(x) := \kappa \left( \frac{m_E}{m_J} \sum_{j\in\A} \tau_j(x)\phi(h_j) - \lambda \sum_{j\in\B} \ttau_j(x)\phi(h_j)\right) ,
\]
where
 $\tau_j$ and $\ttau_j$ are polynomials of degree $\leq C n$ from Lemmas~\ref{lem5} and \ref{lemnew5}, respectively,
$\lambda$ is chosen so that
\be \label{lambda}
Q_n(1) = \frac{m_E}{m_J} \sum_{j\in\A}  \phi(h_j) - \lambda \sum_{j\in\B}  \phi(h_j) =0 ,
\ee
and $\kappa$ is to be prescribed.

We will now show that $\lambda$ is bounded by a constant independent of $m_E/m_J$.

Let $\widetilde E\subset E$ be the subinterval of $E$ such that
\begin{itemize}
\item[(i)] $\widetilde E$ is a union of $\lfloor m_E/3 \rfloor$ intervals $I_j$, and
\item[(ii)]  $\widetilde E$ is  centered at $0$ as much as $E$   allows it, \ie among all subintervals of $E$ consisting of $\lfloor m_E/3 \rfloor$ intervals $I_j$, the center of $\widetilde E$ is closest to $0$.
\end{itemize}

Then, using the fact that the lengths of $|I_i|$ in the Chebyshev partition are increasing toward the middle of $[-1,1]$ and are decreasing toward the endpoints, we conclude that every interval  $I_j$ inside $\widetilde E$ is not smaller than any interval  $I_i$ in $E\setminus \widetilde E$, \ie
\be\label{compare}
{\rm if}\quad I_j \subset \widetilde E\quad{\rm and}\quad I_i \subset E\setminus \widetilde E,\quad{\rm then}\quad |I_j|\geq |I_i|.
\ee
Moreover, we will now show that all intervals $I_j$ inside $\widetilde E$ have  about the same lengths.

We use the following result (see \cite{LS2002}*{Lemma 5} which, unfortunately, contains an inadvertent omission in the conditions for \cite{LS2002}*{(4.6)}):
\begin{quote}
If $0\leq j_1 \leq i < j_2\leq n$, then
\be \label{alleq}
\frac{j_2-j_1}{2}  \leq \frac{x_{j_1}-x_{j_2}}{x_i-x_{i+1}} \leq (j_2-j_1)^2 .
\ee
Moreover, if, in addition, either $2i+1\le j_2+j_1$ and $j_2\leq 3j_1$, or $2i+1>j_2+j_1$ and $n-j_1 \leq 3(n-j_2)$,
then
\be \label{middleeq}
\frac{j_2-j_1}{2} \leq \frac{x_{j_1}-x_{j_2}}{x_i-x_{i+1}} \leq 2(j_2-j_1) .
\ee
In particular, if both inequalities
\be \label{conditions}
 j_2\leq 3j_1  \andd  n-j_1 \leq 3(n-j_2)
 \ee
are satisfied, then \ineq{middleeq} holds.
 \end{quote}

Suppose that $\widetilde E = [x_{i^*}, x_{i_*}]$. Then $i^*-i_* =  \lfloor m_E/3 \rfloor$. Now, if $0\in \widetilde E$, then $i_* \leq n/2 \leq i^*$, and so
$3i_*-i^* = 2i^* - 3  \lfloor m_E/3 \rfloor \geq n -3  \lfloor m_E/3 \rfloor \geq 0 $ and
$3(n-i^*)-(n-i_*) = 2n-3\lfloor m_E/3 \rfloor - 2i_* \geq 0$. Therefore, conditions \ineq{conditions} are satisfied.

If $0\not\in \widetilde E$, then either $E\subset (0,1]$ or $E\subset [-1,0)$ and so, in particular, $m_E\leq \lfloor n/2 \rfloor$. Suppose that $E\subset (0,1]$ (the other case can be dealt with by symmetry).
Then $i^*=j^* < n/2$ and $i_* = j^* - \lfloor m_E/3 \rfloor = j_*+m_E-1 - \lfloor m_E/3 \rfloor \geq m_E - \lfloor m_E/3 \rfloor \geq 2m_E/3$. Hence,
$3i_*-i^* = 2 i_* - \lfloor m_E/3 \rfloor \geq   m_E/3  \geq 0$ and
$3(n-i^*)-(n-i_*) = 2n - 3 i^* + i_* > n/2 >0 $. Hence, conditions \ineq{conditions} are satisfied in this case as well.

Using \ineq{middleeq} we now conclude that
\[
|I_j| \sim   \frac{|\widetilde E|}{m_E}, \quad \mbox{\rm for all }\; I_j \subset \widetilde E .
\]
Now, denote $\widetilde \E  :=   \left\{ 1\leq j \leq n \st I_j \subset \widetilde E\right\}$. Since $\# \widetilde \E =  \lfloor m_E/3 \rfloor$,   for
 $\widetilde \B := \B \cap \widetilde\E = \widetilde\E \setminus\A$, we have
\[
\# \widetilde \B\ge \#\widetilde \E - \#\A \ge \lfloor m_E/3 \rfloor -m_J-2\ge m_E/3-m_E/4-3\ge m_E/20.
\]

Therefore,
\[
\sum_{j\in\B}  \phi(h_j) \geq \sum_{j\in\widetilde \B}  \phi(h_j) \sim \#\widetilde\B \cdot \phi\left( |\widetilde E|/m_E\right) \sim m_E \cdot \phi\left( |\widetilde E|/m_E\right) ,
\]
and since by \ineq{compare},
\[
\sum_{j\in\A}  \phi(h_j) \leq c\, \#\A \cdot \phi\left( |\widetilde E|/m_E\right) \sim m_J \cdot \phi\left( |\widetilde E|/m_E\right) ,
\]
we conclude that
\[
0 < \lambda  \leq   c \frac{m_E}{m_J} \cdot \frac{ m_J \cdot \phi\left( |\widetilde E|/m_E\right) }{m_E \cdot \phi\left( |\widetilde E|/m_E\right)} \sim 1 ,
\]
\ie $\lambda$ is bounded by a constant independent of $m_E/m_J$.

Now, for any $x\in J\cup([-1,1]\setminus E)$ (as well as for any $x\in I_{j_*} \cup I_{j^*}$), taking into account that $\ttau_j'(x)\leq 0$ for all $j\in\B$,  and using \lem{lem5},  \ineq{lowphi} and \ineq{auxsum}  we have
\begin{eqnarray*}
Q_n'(x) & \geq & \kappa   \frac{m_E}{m_J} \sum_{j\in\A} \tau_j'(x)\phi(h_j) \\
& \geq &
C \kappa \delta^{8\alpha}(x) \frac{m_E}{m_J} \sum_{j\in\A} \phi(h_j) h_j^{-1}    \psi_j^{30(\alpha+\beta)}  \\
& \geq &
C \kappa \delta^{8\alpha}(x) \frac{m_E}{m_J} \frac{\phi(\rho)}{\rho}  \sum_{j\in\A} \psi_j^{30(\alpha+\beta)+2k+1}  \\
&\geq &
C \kappa \delta^{8\alpha}(x) \frac{m_E}{m_J} \frac{\phi(\rho)}{\rho}  \sum_{j\in\A} \left( \frac{\rho}{\rho+|x-x_j|} \right)^{60(\alpha+\beta)+4k+2} \\
& \geq &
C \kappa \delta^{8\alpha}(x) \frac{m_E}{m_J} \frac{\phi(\rho)}{\rho} \left(\frac{\rho }{\max\{\rho ,\dist (x,E)\}}\right)^{60(\alpha+\beta)+4k+2} ,
\end{eqnarray*}
 since, for $x\not\in E$, $\max\{\rho, \dist(x, E)\} \sim \min \left\{ |x-x_{j^*}|,|x-x_{j_*}|\right\} + \rho$, and, for $x\in J$, $x\in I_j$ for some $j\in A$, and so
 $\rho/(|x-x_j|+\rho) \sim 1$ for that $j$.

If $x\in E\setminus J$ and $x\not\in I_{j_*} \cup I_{j^*}$,  then  there exists $j_0\in \B$ such that $x\in I_{j_0}$. Hence,
\begin{eqnarray*}
Q_n'(x) & \geq & - \kappa \lambda \ttau_{j_0}'(x) \phi(h_{j_0})
  \geq   - C \kappa  h_{j_0}^{-1} \delta^\alpha \psi_{j_0}^\beta \phi(h_{j_0}) \\
  & \geq &
  - C \kappa  \frac{\phi(\rho)}{\rho} \delta^\alpha \geq - \frac{\phi(\rho)}{\rho} \delta^\alpha ,
\end{eqnarray*}
for sufficiently small $\kappa$.

We now estimate $|Q_n(x)|$.
Let
\[
L(x) := \kappa \left( \frac{m_E}{m_J} \sum_{j\in\A} \chi_j(x)\phi(h_j) - \lambda \sum_{j\in\B} \chi_j(x)\phi(h_j)\right) .
\]
Then, by virtue of \ineq{tauj}, \ineq{newtauj}, \ineq{upperphi} and $\psi_j^2  \leq c \rho(|x-x_j|+\rho)^{-1}$, we have
\begin{eqnarray*}
\lefteqn{|Q_n(x)-L(x)|}\\
 & =& \kappa  \left| \frac{m_E}{m_J} \sum_{j\in\A} \left( \tau_j(x) -\chi_j(x) \right) \phi(h_j) - \lambda \sum_{j\in\B} \left( \ttau_j(x) -\chi_j(x) \right) \phi(h_j)\right| \\
& \leq & C m_E \delta^\alpha \sum_{j\in\E} \phi(h_j) \psi_j^\beta
 \leq
C m_E \delta^\alpha \phi(\rho) \sum_{j\in\E} \psi_j^{\beta-k}  \\
& \leq & C m_E \delta^\alpha \phi(\rho) \sum_{j\in\E} \frac{h_j}{\rho} \psi_j^{\beta-k-2}   \\
&\leq& C m_E \delta^\alpha \phi(\rho) \sum_{j\in\E} \frac{h_j}{\rho}  \left( \frac{\rho}{|x-x_j|+\rho} \right)^{(\beta-k-2)/2} \\
&\leq& C m_E \delta^\alpha \phi(\rho) \sum_{j\in\E}   \frac{h_j \rho}{(|x-x_j|+\rho)^2},
\end{eqnarray*}
provided $(\beta-k-2)/2 \geq 2$.

Hence, it remains to estimate $|L(x)|$. First assume that $x\not\in E$. If $x\le x_{j^*}$, then $\chi_j(x)=0$, $j\in\A \cup\B$, and $L(x)=0$. If, on the other hand, $x>x_{j_*}$, then $\chi_j(x)=1$, $j\in\A \cup\B$, so that \ineq{lambda} implies that $L(x)= 0$.   Hence, in particular, $L(x)=0$ for $x\in I_1\cup I_n$.

Suppose now that  $x\in E\setminus I_1$ (recall that we already assumed that $E$ does not contain $I_n$). Then,
\ineq{alleq} implies that, for all $j\in\E$, $h_j \leq c |E|/m_E \leq c \rho m_E$ (since, again by \ineq{alleq}, it follows that $|E| \leq c \rho m_E^2$), and so
$\phi(h_j) \leq c m_E^k \phi(\rho)$.

Hence, since $\delta = 1$ on $[x_{n-1}, x_1]$,
\begin{eqnarray*}
|L(x)|& \leq& C\left(\frac{m_E}{m_J} \sum_{j\in\A} \phi(h_j)+\lambda  \sum_{j\in\B} \phi(h_j)\right)\\
&\leq& C m_E^{k+1} \delta^\alpha \phi(\rho) .
\end{eqnarray*}
It remains to note that
\[
1 = |E| \sum_{j\in\E} \frac{h_j}{|E|^2} \leq c |E| \sum_{j\in\E} \frac{h_j}{(|x-x_j|+\rho)^2} \leq c m_E^2 \sum_{j\in\E} \frac{\rho h_j}{(|x-x_j|+\rho)^2},
\]
and the proof is complete.
\end{proof}

\sect{Monotone polynomial approximation of piecewise polynomials}\label{sec5}
All constants $C$ and $C_i$ in this section depend only on $k$ and $\alpha$.

  First, we need the following auxiliary result, the proof of which is similar to that of \cite{LS2002}*{Lemma 12}.

\begin{lemma} \label{newlemmasec10}
Let $k\in\N$, $\phi\in\Phi^k$ and  $S\in\Sigma_{k,n}$ be such that
\be \label{newin1}
b_k(S,\phi) \leq 1 .
\ee
 If   $1\leq \mu, \nu \leq n$ are such that the interval $I_{\mu,\nu}$ contains at least $2k-3$  intervals $I_i$ and points $x_i^*\in (x_i, x_{i-1})$ so that

\be \label{newin2}
\rho_n(x_i^*) \phi^{-1} (\rho_n(x_i^*)) |S'(x_i^*)| \leq 1,
\ee
then, for every $1\leq j\leq n$, we have
\be \label{infin}
\norm{ \rho_n \phi^{-1}(\rho_n) S'}{L_\infty(I_j)} \leq c(k) \left[ (j-\mu)^{4k} + (j-\nu)^{4k} \right] .
\ee
\end{lemma}

\begin{proof}
Clearly, it is enough to prove the lemma for $k\geq 2$ since \ineq{infin}  is trivial if $k=1$. Fix $1\leq j\leq n$.
Since every polynomial piece of $S$ has degree $\leq k-1$, it follows from \ineq{newin1} that, for every $1\leq i \leq n$,
\[
\norm{p_i'-p_j'}{I_i} \leq c h_i^{-1}
\norm{p_i-p_j}{I_i}
 \leq c  h_i^{-1} \phi(h_j) \left(\frac{h_{i,j}}{h_j}\right)^k .
\]
Thus, using $h_j^2 \leq c h_i h_{i,j}$, that follows from \ineq{rho1},
and $\phi(h_i) \leq\phi(h_j) \left({h_{i,j}}/{h_j}\right)^k$,
we have, for $x_i^*\in (x_i, x_{i-1})$ for which \ineq{newin2} holds,
\begin{eqnarray*}
|p_j'(x_i^*)| &\leq & c h_i^{-1} \phi(h_j) \left(\frac{h_{i,j}}{h_j}\right)^k + \rho_n^{-1}(x_i^*)\phi  (\rho_n(x_i^*)) \\
& \leq &  c h_i^{-1} \left( \phi(h_j) \left(\frac{h_{i,j}}{h_j}\right)^k + \phi(h_i) \right)
 \leq
 c h_i^{-1}   \phi(h_j) \left(\frac{h_{i,j}}{h_j}\right)^k \\
 &\leq &
   c h_j ^{-1}   \phi(h_j)\left(\frac{h_{i,j}}{h_j}\right)^{k+1} .
\end{eqnarray*}
Since \ineq{alleq} implies that
\[
\frac{h_{i,j}}{h_j} \leq c \left( |i-j| +1 \right)^2 ,
\]
we conclude that
\[
|p_j'(x_i^*)| \leq c   h_j ^{-1}   \phi(h_j)\left( |i-j| +1 \right)^{2k+2} .
\]
We now use the fact that there are $k-1$ points $(x_{i_l}^*)_{l=1}^{k-1}$ with any two of them separated by at least one interval $I_i \subset I_{\mu,\nu}$.

For any  $x\in (x_j, x_{j-1})$, we represent $p_j'$ (which is a polynomial of degree $\leq k-2$) as
\[
p_j'(x) = \sum_{l=1}^{k-1} p_j'(x_{i_l}^*) \prod_{1\leq m\leq k-1, m\neq l} \frac{x-x_{i_m}^*}{x_{i_l}^*-x_{i_m}^*} ,
\]
estimate
\[
\left| \frac{x-x_{i_m}^*}{x_{i_l}^*-x_{i_m}^*} \right| \leq c \frac{h_{j, i_m}}{h_{i_m}} \leq c \left(|j-i_m|+1\right)^2 \leq c \left( (j-\mu)^2 + (j-\nu)^2 \right),
\]
and obtain
\begin{eqnarray*}
 \rho_n(x) \phi^{-1} (\rho_n(x)) |S'(x)|  & \leq  &  c h_j  \phi^{-1} (h_j)) |p_j'(x)| \\
 &\leq &
 c \sum_{l=1}^{k-1} \left( |j-i_l|+1 \right)^{2k+2} \left( (j-\mu)^2 + (j-\nu)^2 \right)^{k-2} \\
 &\leq &
 c \left( (j-\mu)^2 + (j-\nu)^2 \right)^{2k-1} ,
\end{eqnarray*}
which implies \ineq{infin}.
\end{proof}

\begin{theorem}\label{step111} Let  $k,r\in\N$, $k\geq r+1$,  and let $\phi\in\Phi^k$ be of the form $\phi(t):=t^r\psi(t)$, $\psi\in\Phi^{k-r}$. Also, let  $d_+\ge0$, $d_-\ge0$ and $\alpha\ge0$  be given. Then there is a number $\NN=\NN(k,r,\phi,d_+,d_-,\alpha)$  satisfying the following assertion. If  $n\ge \NN$ and $S\in\Sigma_{k,n}\cap \C[-1,1] \cap\Delta^{(1)}$ is such that
\be\label{d1}
b_k(S,\phi) \le 1,
\ee
and, additionally,
\begin{eqnarray}
\label{d+}
\text{if $d_+>0$, then }  & &  d_+|I_2|^{r-1}\le\min_{x\in I_2}S'(x), \\
\label{d+is0}
\text{if $d_+=0$, then }  & &  S^{(i)}(1) =0, \; \text{for all }\;  1\leq i \leq k-2,    \\
\label{d-}
\text{if $d_->0$, then }  & & d_-|I_{n-1}|^{r-1}\le\min_{x\in I_{n-1}}S'(x), \\
\label{d-is0}
\text{if $d_-=0$, then }  & &  S^{(i)}(-1) =0, \; \text{for all }\;  1\leq i \leq k-2,
\end{eqnarray}
 then there exists  a polynomial
$P\in\Delta^{(1)}\cap\Poly_{Cn}$ satisfying, for all $x\in[-1,1]$,
\begin{eqnarray}
\label{approx10}
 & |S(x)-P(x)|\le C\,\delta_n^{\alpha}(x)\phi(\rho_n(x)),   &  \text{if $d_+>0$ and $d_->0$,}   \\
 \label{newapprox10}
 & |S(x)-P(x)|\le C\,  \delta_n^{\min\{\alpha, 2k-2\}} (x)   \phi(\rho_n(x)),   &  \text{if $\min\{d_+, d_-\} = 0$.}
\end{eqnarray}

\end{theorem}

\begin{proof}
Throughout the proof, we fix $\beta :=k+6$ and   $\gamma:=60(\alpha+\beta)+4k+1$. Hence, the constants $C_1,\dots,C_6$ (defined below)  as well as the constants $C$, may depend only on $k$ and $\alpha$. Note that $S$ does not have to be differentiable at the Chebyshev knots $x_j$. Hence, when we write $S'(x)$ (or $S_i'(x)$, $1\leq i \leq 4$) everywhere in this proof, we implicitly assume that $x\neq x_j$, $1\leq j\leq n-1$. Also, recall   that  $\rho := \rho_n(x)$ and $\delta := \delta_n(x)$.

Let $C_1 :=C$, where the constant   $C$ is taken from \ineq{Qu} (without loss of generality we assume that $C_1\leq 1$), and let $C_2:=C$  with $C$ taken from \ineq{7.24'} with $q=1$.
 We also fix an integer $C_3$ such that
\be\label{8.4}
C_3\ge 8k/C_1 .
\ee
Without loss of generality, we may assume that  $n$ is divisible by $C_3$, and put $n_0:=  n/C_3$.

We divide $[-1,1]$ into $n_0$ intervals
\[
E_q:=[x_{qC_3},x_{(q-1)C_3}]=I_{qC_3}\cup\dots\cup I_{(q-1)C_3+1}, \quad 1\leq q\leq n_0 ,
\]
consisting of $C_3$ intervals $I_i$ each (\ie $m_{E_q} = C_3$, for all $1\leq q\leq n_0$).

We   write ``$j\in UC$'' (where ``$UC$'' stands for ``\underline{U}nder \underline{C}ontrol")  if there is
$x_j^*\in(x_j,x_{j-1})$  such that
\be\label{8.5}
S'(x_j^*)\le \frac{5C_2\phi(\rho_n(x_j^*))}{\rho_n(x_j^*)} .
\ee

We   say that $q\in G$ (for ``\underline{G}ood), if the interval $E_q$ contains at least $2k-3$
intervals $I_j$ with $j\in UC$.
Then,   \ineq{8.5} and \lem{newlemmasec10}
imply that,
\be\label{8.7}
S'(x)\le \frac{C\phi(\rho)}{\rho},\quad x\in E_q , \; q\in G.
\ee
Set
\[
E:=\cup_{q\notin G}E_q,
\]
and decompose $S$ into a ``small" part and a ``big" one, by setting
$$
s_1(x):=\begin{cases}  S'(x),&\quad\text{if}\quad x\notin E,\\
0,&\quad\text{otherwise},\end{cases}
$$
and
\[
s_2 (x):=S'(x)-s_1(x) =
\begin{cases}
0,&\quad\text{if}\quad x\notin E,\\
 S'(x),&\quad\text{otherwise},
 \end{cases}
\]
and putting
\[
S_1(x) :=S(-1)+\int_{-1}^xs_1(u)du \andd
S_2(x) :=\int_{-1}^xs_2(u)du.
\]
(Note that $s_1$ and $s_2$ are well defined for $x\neq x_j$,
$1\le j\le n-1$, so that $S_1$ and $S_2$ are well defined
everywhere and possess  derivatives   for $x\neq x_j$,
$1\le j\le n-1$.)

Evidently,
$$
S_1,S_2\in\Sigma_{k,n},
$$
and
\[
S'_1(x)\ge0\,\text{ and }\,S'_2(x)\ge0, \quad x\in[-1,1].
\]
Now, \ineq{8.7} implies that
\[
S'_1(x)\le \frac{C\phi(\rho)}{\rho},\quad x\in[-1,1],
\]
which, in turn, yields by \lem{important},
\[
b_k(S_1, \phi)\le C.
\]
Together with \ineq{d1}, we obtain
\be\label{8.8}
b_k(S_2, \phi)\le b_k(S_1, \phi) + b_k(S, \phi) \le
C+1\le\lceil C+1\rceil=:C_4.
\ee

The set $E$ is a union of disjoint intervals $F_p=[a_p,b_p]$,
between any two of which, all intervals $E_q$ are with $q\in G$.
We may assume that $n>C_3C_4$, and write $p\in AG$ (for ``\underline{A}lmost
\underline{G}ood"), if $F_p$ consists of no more than $C_4$ intervals $E_q$,
 that is, it consists of no more than $C_3C_4$ intervals
$I_j$. Hence, by \lem{newlemmasec10} (with $\mu$ and $\nu$ chosen so that $I_{\mu,\nu}$ is the union of such an interval $F_p$, $p\in AG$,  and one of the adjacent intervals $E_q$ with $q\in G$),
\be\label{4.7}
S_2'(x)\le\frac {C\,\phi(\rho)}{\rho},\quad x\in F_p,\;  p\in AG.
\ee

One may   think of intervals $F_p$, $p\not\in AG$, as  ``long'' intervals where $S'$  is   ``large'' on many subintervals $I_i$  and  rarely dips down to $0$.
Intervals $F_p$, $p \in AG$, as well as all intervals $E_q$ which are not contained in any $F_p$'s (\ie all ``good'' intervals $E_q$)  are where $S'$ is ``small' in the sense that the inequality $ S'(x)\le  {C\phi(\rho)}/{\rho}$ is valid there.

Set
\[
F:=\cup_{p\notin AG}F_p,
\]
note that $E =  \cup_{p\in AG}F_p \cup F$,
and decompose $S$ again by setting
$$
s_4:=\begin{cases} S'(x),&\quad\text{if}\quad x\in F,\\
0,&\quad\text{otherwise},
\end{cases}
$$
and
\[
s_3(x) := S'(x)-s_4(x) =
\begin{cases} 0,&\quad\text{if}\quad x\in F,\\
S'(x),&\quad\text{otherwise},
\end{cases}
\]
and putting
\be\label{s34}
S_3(x) :=S(-1)+\int_{-1}^xs_3(u)du \andd
S_4(x) :=\int_{-1}^xs_4(u)du.
\ee
Then, evidently,
\be\label{8.9}
S_3,S_4\in\Sigma_{k,n}, \quad S_3+S_4 = S,
\ee
and
\be\label{8.10}
S'_3(x)\ge0\,\text{ and }\,S'_4(x)\ge0, \quad x\in[-1,1].
\ee

We remark that, if $x\not\in  \cup_{p\in AG} F_p$, then $s_1(x)=s_3(x)$ and $s_2(x)=s_4(x)$. If
$x \in  \cup_{p\in AG} F_p$, then $s_1(x)=s_4(x)=0$ and $s_2(x)=s_3(x)=S'(x)$.

For $x\in\cup_{p\in AG} F_p$, \ineq{4.7} implies that
$$
S_3'(x)=S_2'(x)\le\frac {C\,\phi(\rho)}{\rho}.
$$
For all other $x$'s,
$$
S_3'(x)=S_1'(x)\le\frac {C\,\phi(\rho)}{\rho}.
$$
We conclude that
\be\label{8.12}
S_3'(x)\le \frac {C_5\,\phi(\rho)}{\rho},\quad x\in[-1,1],
\ee
which by virtue of   \lem{important}, yields that $b_k(S_3,\phi)\le C$. As above, we obtain
\be\label{8.13}
b_k(S_4,\phi)  \leq b_k(S_3,\phi) + b_k(S,\phi) \le C+1\le\lceil C+1\rceil=:C_6.
\ee

We will approximate $S_3$ and $S_4$ by nondecreasing polynomials that achieve the required degree of pointwise approximation.

{\bf Approximation of $S_3$:}

If $d_+>0$, then there exists $\NN^*\in\N$, $\NN^* = \NN^* (d_+, \psi)$, such that, for $n> \NN^*$,
\[
\frac {\phi(\rho)}{\rho}=\rho^{r-1}\psi(\rho)<\frac{d_+|I_2|^{r-1}}{C_5}\le C^{-1}_5\,S'(x),\quad x\in I_2,
\]
where the first inequality follows since $\psi(\rho) \leq \psi(2/n) \to0$  as $n\to\infty$, and the second inequality follows by \ineq{d+}.
Hence, by \ineq{8.12}, if $n>\NN^*$, then $s_3(x)\ne S'(x)$ for $x\in I_2$.  Therefore, since $s_3(x)=S'(x)$, for all $x\notin F$, we conclude that $I_2\subset F$, and so   $E_1\subset F$,
and $s_3(x)=0$, $x\in E_1$. In particular, $s_3(x)\equiv0$, $x\in I_1$.

Similarly, if $d_->0$, then  using \ineq{d-} we conclude that there exists $\NN^{**}\in\N$, $\NN^{**} = \NN^{**}(d_-,\psi)$, such that, if  $n>\NN^{**}$, then $s_3(x)\equiv0$ for all $x\in I_n$.

Thus, when both $d_+$ and $d_-$ are strictly positive, we conclude that for $n\ge\max\{\NN^*,\NN^{**}\}$, we have
\be \label{endpoints}
s_3(x)=0,\quad \mbox{\rm for all }\; x\in I_1\cup I_n.
\ee

Therefore, in view of \ineq{8.9} and \ineq{8.10}, it follows by   \lem{stepeleven} combined with \ineq{8.12} that, in the case $d_+>0$ and $d_->0$, there exists
a nondecreasing polynomial $r_n \in \Poly_{Cn}$   such that
\be\label{8.14}
|S_3(x)-r_n(x)|\le C\,\delta^{\alpha} \phi(\rho),\quad x\in[-1,1].
\ee

Suppose now that $d_+=0$ and $d_->0$. First, proceeding as above, we conclude that $s_3\equiv 0$ on $I_n$.
Additionally, if $E_1 \subset F$, then, as above,  $s_3\equiv 0$ on  $I_1$ as well. Hence, \ineq{endpoints} holds which, in turn, implies \ineq{8.14}.

 If $E_1 \not\subset F$, then $s_3(x)=S'(x)$, $x\in I_1$, and so it follows
 from  \ineq{d+is0} that, for some constant $c_*\geq 0$,
\[
s_3(x)=S'(x)   =     c_*   (1-x)^{k-2}, \quad x\in I_1 .
\]
By \ineq{8.12}   we conclude that
\[
c_* \le C_5 \frac{\phi(\rho_n(x_1))}{(1-x_1)^{k-2} \rho_n(x_1)} \sim  n^{2k-2} \phi(n^{-2}) .
\]
 Hence, for $x\in I_1$,
 \[
 S_3'(x)=s_3(x) \leq C\left(n \varphi(x) \right)^{2k-4}  n^2 \phi(n^{-2}) \leq C\delta^{2k-4}   \frac{\phi(\rho)}{\rho}
 \]
 and
 \[
0\leq  S_3(1) -  S_3(x) = \int_x^1 s_3(u)\, du \leq c (1-x)^{k-1} n^{2k-2} \phi(n^{-2}) \leq C\delta^{2k-2} \phi(\rho).
 \]
We now define
\[
\widetilde S_3(x) :=
\begin{cases}
S_3(x), & \text{if $x<x_1$}, \\
S_3(1), & \text{if $x\in [x_1,1]$}.
\end{cases}
\]
Then $\widetilde S_3 \in \Sigma_{k,n} \cap \Delta^{(1)}$, $\widetilde S_3'(x) \leq C\rho^{-1}\phi(\rho)$, $x\not\in \{x_j\}_{j=1}^{n-1}$,  and  $\widetilde S_3'\equiv 0$ on $I_1\cup I_n$. Note also that $\widetilde S_3$ may be discontinuous at $x_1$ but the jump is bounded by $\phi(\rho_n(x_1))$ there. Hence, \lem{stepeleven} implies that there exists a nondecreasing polynomial
$r_n \in \Poly_{Cn}$   such that
\[
|\widetilde S_3(x)-r_n(x)|\le C\,\delta^{\alpha} \phi(\rho),\quad x\in[-1,1].
\]
Now, since
\[
\left| \widetilde S_3(x) - S_3(x) \right| \leq C\delta^{2k-2} \phi(\rho), \quad x\in [-1,1] ,
\]
we conclude that
\be \label{tildeaux}
| S_3(x)-r_n(x)|\le C\,\delta^{\min\{\alpha, 2k-2\}} \phi(\rho), \quad x\in[-1,1].
\ee

Finally, if $d_-=0$ and $d_+>0$, then the considerations are completely analogous and, if $d_-=0$ and $d_+=0$, then $\widetilde S_3$ can be modified further on $I_n$ using \ineq{d-is0} and the above argument.

Hence, we've constructed a nondecreasing polynomial $r_n \in \Poly_{Cn}$   such that, in the case when both $d_+$ and $d_-$ are strictly positive, \ineq{8.14} holds, and \ineq{tildeaux} is valid if at least one of these numbers is $0$.

{\bf Approximation of $S_4$:}

Given a set $A\subset  [-1,1]$, denote
\[
A^e:=\cup_{I_j\cap A\ne\emptyset}I_j \andd  A^{2e}:=(A^e)^e,
\]
where $I_0=\emptyset$ and $I_{n+1}=\emptyset$.  For example, $[x_7,x_3]^e=[x_8,x_2]$, $I_1^e=I_1\cup I_2$, etc.

Also, given subinterval  $I\subset [-1,1]$ with its endpoints at the Chebyshev knots, we   refer to the right-most and the left-most intervals $I_i$ contained in $I$ as $EP_+(I)$ and $EP_-(I)$, respectively (for the ``\underline{E}nd \underline{P}oint'' intervals). More precisely, if $1\leq \mu<\nu \leq n$ and
\[
I = \bigcup_{i=\mu}^\nu I_i ,
\]
 then $EP_+(I) := I_\mu$, $EP_-(I) := I_\nu$ and $EP(I):= EP_+(I)  \cup EP_-(I) = I_\mu \cup I_\nu$. For example, $EP_+[-1,1] := I_1$, $EP_-[-1,1] := I_n$, $EP_+[x_7, x_3] = [x_4, x_3] = I_4$, $EP_-[x_7, x_3] = [x_7, x_6] = I_7$,
$EP [x_7, x_3] = I_4\cup I_7$,  etc.
Here, we simplified the notation by using $EP_\pm [a,b] := EP_\pm ([a,b])$ and $EP  [a,b] := EP ([a,b])$.

In order to approximate $S_4$, we observe that for $p\notin AG$,
\[
S_4'(x)=S_2'(x),\quad x\in F^{2e}_p,
\]
so that by virtue of \ineq{8.8}, we conclude that
\be\label{8.16}
b_k(S_4,\phi, F^{2e}_p)=b_k(S_2, \phi, F^{2e}_p)\le b_k(S_2, \phi)\le C_4.
\ee
(Note that, for $p\in AG$, $S_4$ is  constant in $F_p^{2e}$ and so $b_k(S_4,\phi, F_p^{2e})=0$.)

We will approximate $S_4$ using the polynomial $D_{n_1}(\cdot,S_4) \in \Poly_{Cn_1}$  defined  in \lem{uncon}  (with $n_1:=C_6n$), and then we construct
  two ``correcting'' polynomials $\overline Q_n, M_n \in\Poly_{Cn}$ (using  \lem{QM}) in order to make sure that the resulting approximating polynomial is nondecreasing.

We begin with $\overline Q_n$. For each $q$ for which $E_q\subset F$, let
$J_q$ be the union of all intervals $I_j\subset E_q$ with $j\in
UC$ with the union of both intervals $I_j\subset E_q$  at the endpoints of $E_q$.
In other words,
\[
J_q := \bigcup_j \left\{ I_j \st j\in UC \andd  I_j\subset E_q\right\}  \;  \cup \;  EP(E_q) .
\]

Since $E_q\subset F$, then $q\notin G$ and so the number   of
 intervals $I_j\subset E_q$ with $j\in UC$  is at most $2k-4$.
Hence, by \ineq{8.4},
\[
m_{J_q}\le 2k-2<  2k   \le \frac{C_1C_3}4\le\frac{C_3}4,
\]
Recalling that the total number $m_{E_q}$ of intervals $I_j$ in $E_q$ is $C_3$ we conclude that \lem{QM} can be used with $E:= E_q$ and $J:=J_q$.
Thus, set
\[
\overline Q_n:=\sum_{q\, :\; E_q\subset F} Q_n(\cdot,E_q,J_q),
\]
where $Q_n$ are   polynomials from \lem{QM}, and denote
\[
J:=\bigcup_{q\, :\; E_q\subset F}J_q.
\]
Then, \ineq{Qu} through \ineq{Qu2} imply that that $\overline Q_n$ satisfies
\begin{align}\label{8.17}
\text{(a)} \qquad \overline Q_n'(x)&\ge 0,\quad x\in[-1,1]\setminus F,\nonumber\\
\text{(b)} \qquad \overline Q_n'(x)&\ge -   \frac{\phi(\rho)}{\rho}\quad x\in F\setminus
J,\\
\text{(c)} \qquad \overline Q_n'(x)&\ge  4  \frac{ \phi(\rho)}{\rho}\delta^{8\alpha},\quad x\in J.\nonumber
\end{align}
Note that  the inequalities in \ineq{8.17} are valid since, for any given $x$,
all relevant $Q_n'(x,E_q,J_q)$, except perhaps one, are nonnegative, and
\[
C_1 \frac{m_{E_q}}{m_{J_q}} \geq \frac{C_1C_3}{2k} \ge 4.
\]
Also, it follows from \eqref{Qu2} that, for any $x\in[-1,1]$,
\begin{eqnarray} \label{8.20}
|\overline Q_n(x)| & \leq & C\delta^\alpha \rho \phi(\rho)    \sum_{q\, :\; E_q\subset F} \sum_{j: \, I_j\subset  E_q}\frac{h_j}{(|x-x_j|+\rho)^2} \\ \nonumber
& \leq &
C\delta^\alpha \rho \phi(\rho)  \sum_{j=1}^n \frac{h_j}{(|x-x_j|+\rho)^2} \\ \nonumber
& \leq &
C\delta^\alpha \rho \phi(\rho) \int_0^\infty  \frac{du}{(u+\rho)^2}    \\ \nonumber
& = &
C\delta^\alpha \phi(\rho).
\end{eqnarray}

Next, we define the polynomial $M_n$. For each $F_p$ with $p\notin AG$, let $J_p^-$ denote the union of the two intervals on the left
side of $F^e_p$ (or just the interval $I_n$ if $-1\in F_p$), and let $J_p^+$ denote
the union of the two intervals on the right side of
$F^e_p$ (or just one interval $I_1$ if  $1\in F_p$), \ie
\[
J_p^- = EP_-(F^e_p) \cup EP_-(F_p) \andd J_p^+ = EP_+(F^e_p) \cup EP_+(F_p).
\]
Also, let $F_p^-$ and $F_p^+$ be the closed intervals each consisting of $m_{F_p^\pm}:=C_3C_4$
intervals $I_j$ and such that $J_p^-\subset F_p^-\subset F^e_p$ and $J_p^+\subset F_p^+\subset F^e_p$, and put
\[
J_p^*:=J_p^-\cup J_p^+ \andd J^*:=\cup_{p\notin AG}J_p^*.
\]
Now, we set
\[
M_n:=\sum_{p\notin AG}\left(Q_n(\cdot,F_p^+,J_p^+) +
Q_n(\cdot,F_p^-,J_p^-)\right).
\]
Since $m_{F_p^+}=m_{F_p^-}=C_3C_4$ and $m_{J_p^+}, m_{J_p^-} \leq 2$, it follows from \ineq{8.4} that
\[
\min\left\{ \frac{m_{F_p^+}}{m_{J_p^+}} ,\frac{m_{F_p^-}}{m_{J_p^-}} \right\} \ge  \frac{C_1 C_3 C_4}{2}
\ge 2C_4.
\]
Then    \lem{QM} implies
\be\label{8.24}
|M_n(x)|\le C\,\delta^{\alpha}\phi(\rho)
\ee
(this follows from \ineq{Qu2} using the same sequence of inequalities that was used to prove \ineq{8.20} above), and
\begin{align}\label{8.21}
\text{(a)} \quad M_n'(x)&\ge-2\frac{\phi(\rho)}{\rho},\quad x\in F\setminus J^*,\nonumber\\
\text{(b)} \quad M_n'(x)&\ge
2C_4\, \delta^{8\alpha}   \frac{\phi(\rho)}{\rho},\quad x\in J^*,\\
\text{(c)} \quad M_n'(x)&\ge
2C_4\, \delta^{8\alpha}\frac{\phi(\rho)}{\rho}\left(\frac\rho{\dist\,(x,F)}
\right)^{\gamma+1},\; x\in[-1,1]\setminus F^e,\nonumber
\end{align}
where in the last inequality we used the fact that
\[
\max\{\rho,\dist\,(x,F^e)\}\le  \dist\,(x,F),\quad
x\in[-1,1]\setminus F^e,
\]
which follows from \ineq{newauxest}.

The third auxiliary polynomial is
  $D_{n_1}:=D_{n_1}(\cdot, S_4)$ with $n_1 = C_6n$ constructed in \lem{uncon}. By \ineq{8.13}, \ineq{7.23'} yields
\be\label{8.25}
|S_4(x)-D_{n_1}(x)|\le C\,\delta^\gamma \phi(\rho) \le C\,\delta^\alpha \phi(\rho)     ,\quad x\in[-1,1],
\ee
since $\gamma > \alpha$, and \ineq{7.24'}  implies that, for any
interval $A\subset [-1,1]$ having  Chebyshev knots as endpoints,
\begin{align}\label{8.26}
|S_4'(x)-D_{n_1}'(x)|&\le C_2\, \delta^{\gamma} \frac{\phi(\rho)}{\rho}b_k(S_4,\phi, A)\\&\quad+
C_2C_6 \, \delta^{\gamma}\frac{\phi(\rho)}{\rho}\frac n{n_1}
\left(\frac{\rho}{\dist(x,[-1,1]\setminus A)}\right)^{\gamma+1},\quad
x\in A.\nonumber
\end{align}
We now define
\be\label{8.27}
R_n:=D_{n_1}+C_2\overline Q_n+C_2M_n.
\ee
By virtue of \ineq{8.20}, \ineq{8.24}, and \ineq{8.25}  we obtain
$$
|S_4(x)-R_n(x)|\le C\,\delta^{\alpha}\phi(\rho),\quad x\in[-1,1],
$$
which combined with \ineq{8.14} and \ineq{tildeaux}, proves \ineq{approx10} and \ineq{newapprox10} for $P :=R_n+r_n$.

Thus, in order to conclude the proof of   \thm{step111}, we should prove that
$P$ is nondecreasing. We recall that $r_n$ is nondecreasing, so   it is sufficient to show that $R_n$ is nondecreasing as well.

Note that \ineq{8.27} implies
\[
R_n'(x) \ge C_2\overline Q_n'(x)+C_2M_n'(x)
 -|S_4'(x)- D_{n_1}'(x)|+S_4'(x), \quad x\in[-1,1],
\]
(this inequality is extensively used in the three cases below),
and that \ineq{8.26} holds for {\em any} interval $A$ with Chebyshev knots as the endpoints, and so we can use
different intervals $A$ for different points $x\in [-1,1]$.
We consider three cases depending on whether (i) $x\in F\setminus J^*$, or (ii)  $x\in J^*$, or (iii)~$x\in[-1,1]\setminus F^e$.

{\bf Case (i):} If $x\in F\setminus J^*$, then, for some $p\notin AG$,
  $x\in F_p\setminus J_p^*$, and so   we take $A:=F_p$. Then, the quotient
inside the  parentheses in \ineq{8.26} is bounded above by 1 (this follows from \ineq{newauxest}). Also, since   $s_4(x)=S'(x)$, $x\in F$, it follows that
$b_k(S_4,\phi, F_p)=b_k(S,\phi, F_p)\le1$. Hence,
\begin{align}\label{8.28}
|S_4'(x)-D_{n_1}'(x)|&\le  C_2\, \frac{\phi(\rho)}{\rho}b_k(S_4,\phi,F_p)+
C_2C_6 \, \frac{\phi(\rho)}{\rho}\frac n{n_1}\\
& \le
2 C_2 \, \frac{\phi(\rho)}{\rho},\quad x\in F\setminus J^*. \nonumber
\end{align}
Note  that $x\notin I_1\cup I_n$ (since   $F\setminus J^*$ does not contain any intervals in $EP(F_p)$, $p\notin AG$), and so $\delta = 1$.

It now follows by  \ineq{8.17}(c), \ineq{8.21}(a),  \ineq{8.28} and \ineq{8.10}, that
\[
R_n'(x)\ge C_2 \, \frac{\phi(\rho)}{\rho}(4-2-2)=0,\quad
x\in J\setminus J^*.
\]
If $x\in F\setminus(J\cup J^*)$, then \ineq{8.5} is violated and so
\[
S_4'(x) = S'(x) >\frac{5C_2\phi(\rho)}{\rho}.
\]
Hence, by virtue of \ineq{8.17}(b), \ineq{8.21}(a) and \ineq{8.28}, we get
$$
R_n'(x)\ge C_2\, \frac{\phi(\rho)}{\rho}(-1-2-2+5)=0,\quad
x\in F\setminus(J\cup J^*).
$$

{\bf Case (ii):}
If $x\in J^*$, then, $x\in J_p^*$, for some $p\notin AG$, and we take  $A:=F_p^{2e}$. Then,
 \ineq{8.16} and \ineq{8.26} imply (again, \ineq{newauxest} is used to estimate the quotient inside the  parentheses in \ineq{8.26}),
\begin{align}\label{8.29}
|S_4'(x)-D_{n_1}'(x)|&\le
C_2\,\delta^{\gamma} \frac{\phi(\rho)}{\rho}b_k(S_4,\phi, F_p^{2e})+
C_2C_6 \,\delta^{\gamma} \frac{\phi(\rho)}{\rho}\frac n{n_1}\\
&\le
C_2C_4\,\delta^{\gamma} \frac{\phi(\rho)}{\rho},\quad x\in J^*.    \nonumber
\end{align}
Now, we note that $EP(F_p) \subset J$, for all $p\notin AG$, and so $F\cap J^* \subset J$. Hence,
 using   \ineq{8.17}(a,c), \ineq{8.21}(b),
 \ineq{8.29} and \ineq{8.10}, we obtain
\[
R_n'(x)\ge 2 C_2 C_4\, \delta^{8\alpha} \frac{\phi(\rho)}{\rho}   - 2 C_2C_4\,\delta^{\gamma} \frac{\phi(\rho)}{\rho} \geq 0 ,
\]
 since $\gamma > 8\alpha$, and so $\delta^{\gamma} \leq \delta^{8\alpha}$.

{\bf Case (iii):}
If $x\in[-1,1]\setminus F^e$, then we take $A$ to be the
connected component of $[-1,1]\setminus F$ that contains $x$. Then by \ineq{8.26},
\begin{align}\label{8.30}
&|S_4'(x)-D_{n_1}'(x)|\nonumber\\
&\quad\le  C_2\, \delta^{\gamma} \frac{\phi(\rho)}{\rho}b_k(S_4,\phi,A)+
C_2C_6 \, \delta^{\gamma} \frac{\phi(\rho)}{\rho}\frac n{n_1}
\left(\frac\rho{\dist(x,[-1,1]\setminus A)}\right)^{\gamma+1}\\
&\quad = C_2 \, \delta^{\gamma} \frac{\phi(\rho)}{\rho}\left(\frac\rho{\dist(x,
F)}\right)^{\gamma+1},\quad x\in[-1,1]\setminus F^e , \nonumber
\end{align}
where we used the fact that $S_4$ is constant in $A$, and so $b_k(S_4,\phi,A)=0$.

Now, \ineq{8.17}(a), \ineq{8.21}(c), \ineq{8.30} and \ineq{8.10} imply,

\[
R_n'(x)\ge
\frac{\phi(\rho)}{\rho} \left(\frac\rho{\dist(x, F)}\right)^{\gamma+1} \left(
2 C_2 C_4 \delta^{8\alpha}   - C_2 \, \delta^{\gamma}
  \right)\geq 0,
\]
 since $C_4\geq 1$ and $\gamma > 8\alpha$.

Thus, $R_n'(x)\geq 0$ for all $x\in[-1,1]$, and so we have constructed a nondecreasing
polynomial $P$, satisfying \ineq{approx10} and \ineq{newapprox10}, for each $n\ge \NN$.
This completes the proof.
\end{proof}

\sect{ Proof of \thm{thm1}}\label{sec555}

In order to prove   \thm{thm1}, we first approximate $f$ by appropriate piecewise polynomials. To this end we make use, among other things, of the following result on pointwise monotone piecewise polynomial approximation (see \cite{LP}).

\begin{theorem}\label{thm2} Given $r\in\N$, there is a constant $c=c(r)$ with the property that if  $f\in \C^r[-1,1]\cap\mon$, then there is a number $\widetilde \NN=\widetilde \NN(f,r)$, depending on $f$ and $r$, such that for $n\ge \widetilde \NN$, there are nondecreasing continuous piecewise polynomials $S\in\Sigma_{r+2,n}$ satisfying
\be\label{interspline}
|f(x)-S(x)|\le c(r)\left(\frac{\varphi(x)}n\right)^r \omega_2\left(f^{(r)},\frac{\varphi(x)}n\right),\quad x\in[-1,1].
\ee
Moreover,
\be\label{interend}
|f(x)-S(x)|\le c(r)\varphi^{2r}(x) \omega_2\left(f^{(r)}, \frac {\varphi(x)}n\right),\quad x\in[-1,x_{n-1}]\cup[x_{1},1].
\ee
\end{theorem}

As was shown in \cite{LP}, near $\pm 1$, polynomial pieces of the spline $S$ from the statement of \thm{thm2} can be taken to be Lagrange-Hermite polynomials of degree $\leq r+1$. Namely,
\[
S\big|_{[x_2,1]}(x)=f(1)+\frac{f'(1)}{1!}(x-1)+\dots+\frac{f^{(r)}(1)}{r!}(x-1)^r+a_+(n;f)(x-1)^{r+1}
\]
and
\[
S\big|_{[-1,x_{n-2}]} (x)=f(-1)+\frac{f'(-1)}{1!}(x+1)+\dots+\frac{f^{(r)}(-1)}{r!}(x+1)^r+a_-(n;f)(x+1)^{r+1},
\]
where constants $a_+(n,f)$ and $a_-(n,f)$   depend  only on $n$ and $f$, and are chosen so that $S(x_2)=f(x_2)$ and $S(x_{n-2})=f(x_{n-2})$.
It was shown in \cite{LP}*{(3.1)} that
\[
|a_+(n,f)|\le\frac1{r!\left(|I_1|+|I_2|\right)}\omega_1(f^{(r)},|I_1|+|I_2|,I_1\cup I_2)
\]
and
\[
|a_-(n,f)|\le\frac1{r!\left(|I_{n-1}|+|I_n|\right)}\omega_1(f^{(r)},|I_{n-1}|+|I_n| ,I_{n-1} \cup I_n).
\]

On $I_{j}$'s with $j\neq 1,2,n-1,n$,   polynomial pieces $p_{j}$ of $S$ were constructed  using  \cite{LS}*{Lemma 2, p. 58}.

For $f\in \C^r[-1,1]$, let $i_+\geq 1$, be the smallest integer $1\leq i \leq r$, if it exists, such that $f^{(i)}(1)\neq 0$, and denote
\[
D_+(r,f) :=
\begin{cases}
(2   r!)^{-1} |f^{(i_+)}(1)|, &    \mbox{\rm if $i_+$ exists,} \\
0, & \mbox{\rm otherwise.}
\end{cases}
\]
Similarly, let let $i_-\geq 1$, be the smallest integer $1\leq i \leq r$, if it exists, such that $f^{(i)}(-1)\neq 0$, and denote
\[
D_-(r,f) :=
\begin{cases}
(2 r!)^{-1} |f^{(i_-)}(-1)|, &    \mbox{\rm if $i_-$ exists,} \\
0, & \mbox{\rm otherwise.}
\end{cases}
\]

Using the above as well as the observation that $|I_1|+|I_2| \to 0$ and $|I_{n-1}|+|I_n| \to 0$ as $n\to \infty$, we can strengthen   \thm{thm2} as follows.

\begin{lemma}\label{spline}Given $r\in\N$, there is a constant $c=c(r)$ with the property that if a function $f\in \C^r[-1,1]\cap\mon$,   then there is an integer $\NN=\NN(f,r)$   depending on $f$ and $r$, such that for $n\ge \NN$, there are nondecreasing continuous piecewise polynomials $S\in\Sigma_{r+2,n}$ satisfying \ineq{interspline}, \ineq{interend},
\be \label{hermite}
S^{(i)}(-1) = f^{(i)}(-1) \andd S^{(i)}(1) = f^{(i)}(1), \quad \text{for all $1\leq i \leq r$,}
\ee
\be\label{der1}
S'(x)\ge D_+(r,f)   (1-x)^{r-1},\quad x\in (x_2,1],
\ee
and
\be\label{der2}
S'(x)\ge D_-(r,f)  (x+1)^{r-1},\quad x\in[-1,x_{n-2}).
\ee
\end{lemma}

\medskip

\begin{proof}[\bf Proof of  \thm{thm1}] Given $r\in\N$ and a nondecreasing $f\in C^{(r)}[-1,1]$, let $\psi\in\Phi^2$
be such that $\omega_2(f^{(r)},t) \sim \psi(t)$,   denote $\phi(t):=t^r\psi(t)$, and note that $\phi\in\Phi^{r+2}$.

For each $n\ge\NN$, we take the piecewise polynomial $S\in\Sigma_{r+2,n}$ of   \lem{spline} and we observe that
\[
\omega_{r+2}(f,t)\le t^r\omega_2(f^{(r)},t) \sim \phi(t),
\]
so that by   \lem{b_k<c} with $k=r+2$, we conclude that
\[
b_{r+2}(S,\phi)\le c := \varsigma.
\]
Now,   it follows from \ineq{der1} and \ineq{rho}  that
\[
\min_{x\in I_2} S'(x)\geq D_+(r,f) |I_1|^{r-1} \geq 3^{-r+1} D_+(r,f) |I_2|^{r-1}
\]
and, similarly, \ineq{der2} yields
\[
\min_{x\in I_{n-1}} S'(x)  \geq 3^{-r+1} D_-(r,f) |I_{n-1}|^{r-1} .
\]
Hence, using \thm{step111} with $k=r+2$,
 $d_+ := \varsigma^{-1} 3^{-r+1} D_+(r,f)$,  $d_- := \varsigma^{-1} 3^{-r+1} D_-(r,f)$ and $\alpha= 2k-2= 2r+2$, and observing that
$b_{r+2}(\varsigma^{-1} S, \phi) \le 1$,  we conclude that there exists a polynomial $P\in \Poly_{cn} \cap \mon$ such that
\be\label{estim}
|S(x)-P(x)|\le c\delta_n^{2r+2} (x) \rho^r_n(x)\psi(\rho_n(x)) ,\quad x\in[-1,1].
\ee

In particular, for $x\in I_1\cup I_n$, $x\neq-1,1$, using the fact that $\rho_n(x) \sim n^{-2}$ for these $x$, and $t^{-2}\psi(t)$ is nonincreasing  we have
\begin{align}\label{ends}
|S(x)-P(x)|&\le c(n\varphi(x))^{2r+2}\rho^r_n(x)\psi(\rho_n(x))\\ \nonumber
& \le c n^2 \varphi^{2r+2}(x)    \left( \frac{n\rho_n(x)}{\varphi(x)}\right)^2 \psi\left(\frac{\varphi(x)}{n} \right) \\ \nonumber
&\le c\varphi^{2r}(x)\omega_2\left(f^{(r)},\frac{\varphi(x)}n\right).
\end{align}
In turn, this implies for $x\in I_1\cup I_n$, that
\[
|S(x)-P(x)|\le c\left(\frac{\varphi(x)}n\right)^r\omega_2\left(f^{(r)},\frac{\varphi(x)}n\right),
\]
which combined with \ineq{estim} implies
\be\label{endss}
|S(x)-P(x)|\le c\left(\frac{\varphi(x)}n\right)^r\omega_2\left(f^{(r)},\frac{\varphi(x)}n\right),\quad x\in[-1,1].
\ee
Finally, \ineq{endss} together with \ineq{interspline} yield \ineq{interpol}, and \ineq{ends} together with \ineq{interend} yield \ineq{interpol1}. The proof of   \thm{thm1} is complete.
\end{proof}


\end{document}